\documentclass[11pt,twoside,a4paper,reqno]{amsart}
\usepackage[utf8]{inputenc}
\usepackage{amsmath}
\usepackage{amssymb}
\usepackage{amsthm}
\usepackage[english]{babel}
\usepackage{tikz-cd}
\usepackage{url}
\usepackage{enumerate}
\usepackage[normalem]{ulem}
\newtheorem{theorem}{Theorem}
\newtheorem{corollary}[theorem]{Corollary}
\newtheorem{definition}[theorem]{Definition}

\newtheorem{lemma}[theorem]{Lemma}
\newtheorem{remark}[theorem]{Remark}
\newtheorem{proposition}[theorem]{Proposition}

\newcommand{\supp}{\text{supp}}
\newcommand{\as}[1]{\textit{as } #1 \rightarrow\infty}
\renewcommand{\d}{\, \text{d}}
\newcommand{\s}[2]{(#1_#2)_{#2 \geq 0}}
\newcommand{\fceil}[2]{\left \lceil \frac{#1}{#2}  \right \rceil}
\newcommand{\ffloor}[2]{\left \lfloor \frac{#1}{#2}  \right \rfloor}

\usepackage{xcolor}
\newcommand{\F}{\mathcal{F}}
\newcommand{\R}{\mathbb{R}}

\newcommand{\C}{\mathbb{C}}
\newcommand{\N}{\mathbb{N}}
\newcommand{\Z}{\mathbb{Z}}
\newcommand{\T}{\mathbb{T}}
\newcommand{\B}{\mathcal{B}}

\usepackage{parskip}
\usepackage[hang,flushmargin]{footmisc}
\usepackage{graphicx}

\begin{document}
\title{Weak mixing and sparse equidistribution}
\author{M. Auer}\thanks{The author thanks the University of Maryland for their hospitality during this work. Special thanks go to my advisor  A. Kanigowski for suggesting this problem and providing helpful comments throughout. Thanks also go to D. Dolgopyat, for supplying his knowledge on many of the examples.}
\date{}
\maketitle
\makeatletter{\renewcommand*{\@makefnmark}{}

\footnotetext{\textit{2020 Mathematics} 37A10, 37A25, 37A44, 37A46, 11L07.}\makeatother}

\begingroup
\leftskip4em
\rightskip\leftskip
\begin{small}
\paragraph*{Abstract}
The celebrated Birkhoff Ergodic Theorem asserts that, for an ergodic map, orbits of almost every point equidistributes when sampled at integer times. This result was generalized by Bourgain to many natural sparse subsets of the integers. On the other hand, our understanding of the behaviour of orbits of \textbf{all} points in a dynamical system is much less understood, especially for sparse subsets of the integers.\\
We generalize a method introduced by A. Venkatesh to tackle this problem in two directions, general $\R^d$ actions instead of flows, and weak mixing, rather than mixing, actions. Along the way, we also establish some basic properties of weak mixing and show weak mixing for the time 1-map of a weak mixing flow.

\end{small}
\par
\endgroup

\part{Results}

\section{Introduction and Notation}

One of the most basic and important results in ergodic theory is the Birkhoff Ergodic Theorem, which states that if $\phi$ is a measure-preserving ergodic map on a state space $(X,\mu)$, then for almost every $x\in X$ the orbit of $x$ equidistributes, i.e. for every continuous $f$ (in fact all $f\in L^1$) and almost every $x\in X$ it holds that
\begin{equation}\label{ergsum}
\frac{1}{N} \sum_{n=0}^{N-1} f(\phi_n(x)) \rightarrow \int_X f \d\mu \;\;\;\as{n}.
\end{equation}
The statement remains true if we replace pointwise convergence by\footnote{Unless there is a risk of confusion, we will not specify the underlying space when writing $L^2$ (or $L^1$).} $L^2$ convergence.

This theorem can be generalized in several directions. One of the most interesting (and challenging) is whether one can replace $\N$ by different sequences, i.e. for a given sequence $\s{b}{n}\subset \N$, when does 
\begin{equation}\label{sparsesum}
\frac{1}{N} \sum_{n=0}^{N-1} f(\phi_{b_n}(x))
\end{equation} 
converge? This question also makes sense if $\s{b}{n}$ is a sequence of not necessarily integers and $\phi$ is a flow.

If one asks for $L^2$ asymptotics, instead of pointwise a.e convergence, much more is known. One can often use spectral methods to study such sparse limit theorems. Classical examples include $b_n=n^2$ or $b_n=p_n$ is the $n$th prime. Pointwise almost everywhere statements are also known. A comprehensive discussion of this topic can be found in \cite{ergharmRW}.

Instead of changing the sequence one sums over, we can alternatively 
put $\{0,1\}$ weights into the ergodic sum \eqref{ergsum}. 
Namely, let $B\!\!=\!\!\{b_n\}$ and $\theta_n\!\!=\!\! 1_B(n)$, then \eqref{sparsesum} is equal to
\begin{equation}\label{weightsum}
\frac{1}{N} \sum_{n=0}^{b_N-1} \theta_n f(\phi_{n}(x)).
\end{equation}
The sums in \eqref{weightsum} are of interest even when $\theta_n$ is not necessarily $\{0,1\}$-valued. For example, $\theta_n=\mu_n$ the Möbius function is of interest for questions on Möbius disjointness.

In the article \cite{venkateshtwist}, A. Venkatesh investigated convergence of \eqref{sparsesum} for \textbf{all} $x$ and $b_n=n^{1+\epsilon}$ when $\phi$ is a flow, thereby giving a partial answer to a conjecture by N. Shah. and G. Margulis.\\ His method relied on polynomial mixing and polynomial equidistribution of the flow. We give a generalisation of the method to $\R^d$ or $\Z^d$ actions.

In addition to proving convergence of \eqref{sparsesum} or \eqref{weightsum} (in a higher dimensional setting), we will also provide estimates on the speed of convergence. This will lead to somewhat cumbersome formulas, so we introduce the following notation for
the ease of reading.

For $A,B\in \R$ we write 
\begin{equation*}
A\ll B \;\; \iff \;\; \exists C>0 \; \textit{ such that } \; A \leq CB,
\end{equation*}
the constant $C>0$. We write $A\ll_{f,g,h} B$ if the implicit constant is allowed to depend on the quantities $f,g,h$. We write $A\sim B$ if $A\ll B$ and $B\ll A$.

In the following, $\phi$ will denote a continuous group action of a group $G$ on a 
metric space $(X,d)$. Unless explicitly stated, we will suppress the dependence on $\phi, G$ and $X$ whenever we write $A\ll B$.

For a vector $\textbf{u}=(u_1,...u_d)\in \R^d$ and $p\geq 1$ we write
\begin{equation*}
|\textbf{u}|^p_p=\sum_{i=1}^d |u_i|^p \; \textit{ and } \; |\textbf{u}|_{\infty}=\max_{j=1,...,d} |u_j|.
\end{equation*}
For $x\in\C$ denote $e(x)=e^{2\pi i x}$.

The main objects of study in this paper will be weak mixing, twisted integrals (both in $L^2$ and pointwise) and pointwise ergodic averages. 

Below, we shall present results of either qualitative or quantitative nature (we present the case of polynomial bounds). For this reason, we shall introduce the objects for the largest class of functions where they make sense ($L^1$ resp. $L^2$). However, quantitative statements (decay rates) are sensitive to smoothness ($C^1$, Hölder, Sobolev, etc.), hence whenever we talk about rates of decay, we have to restrict ourselves to a certain subspace (or subalgebra) of $L^1$ resp. $L^2$. Let us make this precise. 
Let $\B\subset L^2\cap L^1$ be a Banach subalgebra of $L^2\cap L^1$, let's say to $(\B,||\cdot||_{\B})$. We will say that $\B$ satisfies condition (N) if
\begin{itemize}
\item[(N1)] $(\B,||\cdot||)\subset (C^{\rho}, ||\cdot||_{C^{\rho}})$ is a subalgebra of the space of H\"older functions\footnote{Note that for any subalgebra $\B'\subset L^2\cap L^1$ one can choose $\B=\B'\cap C^{\rho}$ and $||\cdot||_{\B}=\max(||\cdot||_{\B'},||\cdot||_{C^{\rho}})$.} for some $\rho>0$, i.e. $\B\subset C^{\rho}$ and for each $f\in \B$ we have $||f||_{C^{\rho}}\ll ||f||_{\B}$,
\item[(N2)] and $\phi$ grows $||\cdot||_{\B}$ only polynomially, i.e. there is a $\mathcal{K}>0$
 such that
\begin{equation}\label{PolyNormGrowth}
||f\circ \phi_{\textbf{t}}||_{\B} \ll ||f||_{\B} |\textbf{t}|^{\mathcal{K}}.
\end{equation}
\end{itemize}

Let us fix some notation, the definitions below are standard and can be found in any book on ergodic theory.

A continuous action $\phi$ of $G=\R^d$ (or $\Z^d$) is called \textit{weak mixing} (w.r.t $\mu$), for some invariant measure $\mu$ on $X$, if, for any $f,g\in L^2_0 = \{f\in L^2(X) \;|\; \int_X f \d\mu=0\}$, we have
\begin{equation*}
\frac{1}{(2T)^d} \int_{[-T,T]^d} | \langle f,g(\phi_{\textbf{t}}) \rangle | \d m(\textbf{t}) \rightarrow 0 \;\;\;\as{T},
\end{equation*} 
where $m$ is the $d$-dimensional Lebesgue-measure (counting measure on $\Z^d$) and
\begin{equation*}
\langle f,g \rangle=\int_X f(x) \overline{g(x)} \d\mu(x).
\end{equation*}

The rates of weak mixing $\alpha^{\phi}=(\alpha^{\phi})_{f,g})$ are given as
\begin{equation}\label{wmratesdef}
\frac{1}{(2T)^d} \int_{[-T,T]^d} | \langle f,g(\phi_{\textbf{t}}) \rangle | \d m(\textbf{t}) = \alpha^{\phi}_{f,g}(T) \;\;\;\forall T>0.
\end{equation} 
If $\alpha^{\phi}_{f,g}(T)\ll ||f||_{\B} ||g||_{\B} T^{-\delta}$ for some $\delta>0$ and $f,g\in L^2_0\cap \B$, then we say that $\phi$ is \textit{polynomially weak mixing}. 

We say that $\phi$ has \textit{decaying ergodic averages} (w.r.t $\mu$) at $x\in X$ and for the function $f\in L^1$ if, 
\begin{equation*}
\left| \frac{1}{(2T)^d} \int_{[-T,T]^d} f(\phi_{\textbf{t}}(x)) \d m(\textbf{t}) - \int_X f \d\mu \right| \rightarrow 0 \;\;\;\as{T}.
\end{equation*}
The rates of decay $\beta^{\phi}=(\beta^{\phi}_{f,x})$ are defined as
\begin{equation}\label{ergratesdef}
\left|\frac{1}{(2T)^d} \int_{[-T,T]^d} f(\phi_{\textbf{t}}(x)) \d m(\textbf{t}) - \int_X f \d\mu\right| = \beta^{\phi}_{f,x}(T).
\end{equation}
If $\beta^{\phi}_{f,x}(T)\ll ||f||_{\B} T^{-\delta}$ for some $\delta>0$ and $f\in L^1 \cap \B$, then we say that $\phi$ has \textit{polynomially decaying ergodic averages} (w.r.t $\mu$) at $x$. If $\phi$ has polynomially decaying ergodic averages (w.r.t $\mu$) at every $x$, then we say that $\phi$ is \textit{polynomially ergodic}. \\
If $\beta^{\phi}_{f,x}(T)\rightarrow 0$ as $T\rightarrow\infty$ for all uniformly\footnote{$\mu$-generic points are more commonly defined as points $x$ such that $\beta^{\phi}_{f,x}(T)\rightarrow 0$ for all continuous (not necessarily uniformly) and bounded $f$. However, these two definitions are equivalent;

If $\beta^{\phi}_{f,x}(T)\rightarrow 0$ as $T\rightarrow\infty$ for $f$ uniformly continuous and bounded, then, by the Portmanteau Theorem (see e.g \cite{elstrodt2013maß}), it holds that $\mu_T \Rightarrow \mu$, where $\mu_n=\frac{1}{(2T)^d} \int_{[-T, T]^d} \delta_{\phi_{\textbf{t}}(x)} \d m(\textbf{t})$. Here $\Rightarrow$ denotes weak convergence. But then it follows that also $\beta^{\phi}_{f,x}(T)\rightarrow 0$ as $T\rightarrow\infty$ for functions that are (not necessarily uniformly) continuous and bounded.} continuous and bounded $f$, then $x$ is said to be $\mu$-generic w.r.t $\phi$. The basin of $\mu$ w.r.t $\phi$ is the collection of all such points and shall be denoted by $B_{\phi}(\mu)$.

We say that $\phi$ has \textit{decaying twisted integrals} at $x\in X$ if, for all $f\in L^2_0$ and $\textbf{a}\in \R^d$ (or $\textbf{a}\in \T^d \tilde{=} [0,1)^d$), we have
\begin{equation}\label{twistdefpoint}
\left|\frac{1}{(2T)^d} \int_{[-T,T]^d} f(\phi_{\textbf{t}}(x)) e( \langle \textbf{a},\textbf{t}\rangle) \d m(\textbf{t})\right| \rightarrow 0 \;\;\;\as{T}.
\end{equation}
The rates of decay $\gamma^{\phi}=(\gamma^{\phi}_{f,x})$ are defined as
\begin{equation}\label{twistratesdef}
\sup_{\textbf{a}\in \R^d} \left|\frac{1}{(2T)^d} \int_{[-T,T]^d} f(\phi_{\textbf{t}}(x)) e( \langle \textbf{a},\textbf{t}\rangle) \d m(\textbf{t})\right| = \gamma^{\phi}_{f,x}(T) \;\;\;\forall T>0.
\end{equation} 
If $\gamma^{\phi}_{f,x}(T)\ll ||f||_{\B} T^{-\delta}$ for some $\delta>0$ and $f\in L^2_0 \cap \B$, then we say that $\phi$ has \textit{polynomially decaying twisted integrals} at $x$.
Analogous definitions and notation will be adapted for $L^2$ convergence, i.e
\begin{equation}\label{twistdefl2}
\sup_{\textbf{a}\in \R^d} \left|\left|\frac{1}{(2T)^d} \int_{[-T,T]^d} f(\phi_{\textbf{t}}(x)) e( \langle \textbf{a},\textbf{t}) \d m(\textbf{t})\right|\right|_{L^2(X)} = \gamma^{\phi}_{f}(T) \;\;\;\forall T>0.
\end{equation}
If there is no risk of confusion, which transformation is meant, we will drop the $\phi$ from the notation, i.e we will write $\alpha=\alpha^{\phi}, \beta=\beta^{\phi}$ and $\gamma=\gamma^{\phi}$.

\begin{remark}\label{rem1}
In the above definitions, maybe it is more familiar to consider one-sided, instead of two-sided, integrals e.g
\begin{equation*}
\left|\frac{1}{T^d} \int_{[0,T]^d} f(\phi_{\textbf{t}}(x)) e( \langle \textbf{a},\textbf{t}) \d m(\textbf{t})\right|.
\end{equation*}
For twisted integrals it makes no difference; we can replace $\textbf{a}$ by $\textbf{a}+(T, T,\dots, T)$. Also, weak mixing is equivalent to the decay of twisted integrals in $L^2$.

We prefer to work with two-sided integrals because in the proof of Proposition \ref{vprop1} a two-sided term will appear even if we start with one-sided. 
\end{remark}

In the following, $\alpha$ and $\beta$ will always refer to the rates defined in \eqref{wmratesdef} and \eqref{ergratesdef}, $\gamma$ will refer to the rates defined in \eqref{twistratesdef} or the corresponding rates in \eqref{twistdefl2}. Whether $m$ refers to the Lebesgue measure on $\R^d$ or counting measure on $\Z^d$ will be clear from the context, whenever there is a risk of confusion we will specify. 

Twisted integrals in $L^2$ are not all that exciting, it is well known that decay in $L^2$ is equivalent to weak mixing, and there is an explicit relation\footnote{See Proposition \ref{wmtwistprop1} and Lemmas \ref{twistspeclem} and \ref{specwmlem}.} between the rates of decay. More explicitly:

\begin{theorem}\label{wmtwistthm}
The action $\phi$ is weakly mixing if and only if twisted integrals decay in $L^2$. Furthermore, the speed of weak mixing is polynomial in some $\B$ if and only if the decay rate for twisted integrals in $L^2$ is.
\end{theorem}

This relation is well known. For completeness’ sake, we shall present a proof in \S \ref{mixtwistsec}.

\section{The restricted action (time 1 map)}

For this section, $G=\R^d$ and $m$ is the $d$-dimensional Lebesgue
measure. The \textit{restricted action} $\phi^{(\textbf{1})}$ is a $\Z^d$ action given by restriction of $\phi$, i.e. $\phi^{(\textbf{1})}_{\textbf{t}}=\phi_{\textbf{t}}$ for $\textbf{t}\in\Z^d$. If $d=1$, then $\phi^{(\textbf{1})}$ is the time-1-map of the flow. Clearly if $\phi$ preserves a measure $\mu$ then so does $\phi^{(\textbf{1})}$, however ergodicity is not preserved under restriction\footnote{Consider $\phi:\T^d \rightarrow \T^d$, given by $\phi_{\textbf{t}}(\textbf{x})=\textbf{x}+\textbf{t}$. Clearly, $\phi$ is (uniquely) ergodic for the Lebesgue measure, but $\phi^{(\textbf{1})}=Id$.}. 

On the other hand, it is clear that if $\phi$ is mixing, then so is $\phi^{(\textbf{1})}$. The same is true for weak mixing; 

Indeed, if $f$ is an eigenfunction for $\phi$, then it is also an eigenfunction for $\phi^{(\textbf{1})}$. Hence, if $\phi^{(\textbf{1})}$ has no non-zero eigenfunctions, then neither does $\phi$.

The converse is also true. Suppose $\phi^{\textbf{1}}$ has a non-zero eigenfunction $f\in L^2_0$, say $f\circ \phi_{\textbf{n}}=e(\langle \textbf{a}, \textbf{n} \rangle) f$ for some $\textbf{a}\in [0,1)^d$ and $\textbf{n}\in \Z^d$. Let $g$ be continuous with $\langle g, f \rangle\not=0$, then\footnote{Note that $\langle g, f \circ \phi_{\textbf{t}} \rangle$ is a continuous function in $\textbf{t}$, even though $f$ is not continuous. To see this use, invariance of the measure
\begin{align*}
|\langle & g, f \rangle - \langle g, f \circ \phi_{\textbf{t}} \rangle| = | \langle g - g\circ \phi_{-\textbf{t}}, f \rangle|\\
&\leq ||g-g\circ \phi_{-\textbf{t}}||_{L^2(X)} ||f||_{L^2(X)} \rightarrow 0\;\;\;\as{\textbf{t}}.
\end{align*}
It follows that $\langle g, f \circ \phi_{\textbf{t}} \rangle \not = 0$ for small $\textbf{t}$.}
\begin{equation*}
\int_{[0,1]^d} |\langle g, f\circ \phi_{\textbf{t}} \rangle | \d \textbf{t} >0,
\end{equation*}
it follows that
\begin{equation*}
\frac{1}{(2T)^d} \int_{[-T,T]^d} |\langle g, f\circ \phi_{\textbf{t}} \rangle | \d \textbf{t} =  \int_{[0,1]^d} |\langle g, f\circ \phi_{\textbf{t}} \rangle | \d \textbf{t} \not\rightarrow 0,
\end{equation*}
and $\phi$ is not weak mixing.

We will provide a quantitative version below in \S \ref{restrictsec}.

\begin{theorem}\label{wmtime1thm}
The action $\phi$ is weak mixing if and only if the restricted action $\phi^{(\textbf{1})}$ is. Furthermore, if $\phi$ is H\"older continuous, and assumption (N) is satisfied, then $\phi$ is polynomially weak mixing if and only if $\phi^{(\textbf{1})}$ is.
\end{theorem}

\section{Pointwise weighted Equidistribution along sparse sequences}

Here, we will investigate equidistribution along sparse sequences. For $d=1$ this means the following, let $\s{a}{n}$ be an increasing sequence of real numbers, for $f\in \B$ the object of interest is
\begin{equation}\label{sparsesumdef}
\frac{1}{N} \sum_{n=0}^{N-1} f(\phi_{a_n}(x)) \;\;\; x\in X.
\end{equation}
Convergence of sums like this has been of increasing interest, for example for $a_n=n^{1+\epsilon}$, with $\epsilon>0$, convergence of the above sum was conjectured by N. Shah when $\phi$ is the horocycle flow. This has been proven (on a compact space), for 
$\epsilon>0$ small enough, in \cite[Theorem 3.1]{venkateshtwist}.

The main idea of the proof is to show that polynomial mixing together with pointwise polynomial decay of ergodic averages implies \textbf{pointwise} polynomial decay of twisted integrals, i.e. there is a $\kappa>0$ such that for $f\in \B\cap L^2_0$ and $a\in\R$ we have 
\begin{equation*}
\frac{1}{T} \left|\int_0^T f(\phi_t(x)) e(at) \d t\right| \ll ||f||_{\B} T^{-\kappa}.
\end{equation*}
It is, in fact, enough to require polynomial weak mixing, rather than polynomial mixing. We will give an elementary proof of this fact, and a generalisation to the situation where $\phi$ is an $\R^d$- or $\Z^d$-action. Furthermore, we shall give an estimate on $\kappa$.

The general bound is the following.

\begin{proposition}\label{vprop1}
For all $f\in L^2_0\cap L^{\infty}$, $\textbf{a}\in \R^d$ (or $[0,1)^d$), $x\in X$ and $T,H>0$ we have
\begin{equation}\label{vprop1rates}
\begin{aligned}
&\left|\frac{1}{(2T)^d} \int_{[-T,T]^d} f(\phi_{\textbf{t}} x) e( \langle \textbf{a} ,\textbf{t} \rangle) \d m(\textbf{t}) \right| \\
& \ll \frac{H}{T} ||f||_{L^{\infty}} + \alpha_{f,f}(H)^{\frac{1}{2}} + \frac{1}{(2H)^d} \left( \int_{[-H,H]^d} \int_{[-H,H]^d} \beta_{f(\phi_{\textbf{h}_1}) \overline{f(\phi_{\textbf{h}_2})},x}(T)\d m(\textbf{h}_1)  \d m(\textbf{h}_2) 
\right)^{\frac{1}{2}}.
\end{aligned}
\end{equation}

\end{proposition}

In the case of polynomial decay, we obtain the following.

\begin{corollary}\label{quantpolytwistcor}
If $\phi$ has polynomially decaying ergodic averages and is polynomially weak mixing, i.e. there are $\delta_1,\delta_2>0$ such that, for $x\in X$ and $f,g\in L^2_0\cap \B$,
\begin{equation*}
\alpha_{f,g}(T)\ll ||f||_{\B}||g||_{\B} T^{-\delta_1} \;\textit{ and } \; \beta_{f,x}(T)\ll ||f||_{\B} T^{-\delta_2},
\end{equation*}
then
\begin{equation*}
\gamma_{f,x}(T) \ll T^{-\delta} ||f||_{\B},
\end{equation*}
where
\begin{equation*}
\delta=\min\left(\frac{\delta_1}{2+\delta_1}, \frac{\delta_1}{2(d+\delta_1)}, \frac{\delta_1\delta_2}{2(\mathcal{K}d+\delta_1)}\right),
\end{equation*}
where $\mathcal{K}$ is the constant from \eqref{PolyNormGrowth}.
In particular $\delta>0$.
\end{corollary}

Pointwise decay of twisted integrals (or sums in the case $G=\Z^d$) shall be used to show either sparse equidistribution (i.e. sums as in \eqref{sparsesumdef}) or weighted equidistribution; for $d=1$ and real numbers $\theta_0,\theta_1,...$ we shall investigate sums of the form
\begin{equation}\label{weightsumdef}
\frac{1}{N} \sum_{n=0}^{N-1} \theta_n f(\phi_n(x)) \;\;\; x\in X.
\end{equation}
Let us define the higher dimensional analogue of \eqref{sparsesumdef} and \eqref{weightsumdef}.

Let $B_0\subset B_1 \subset ... \subset G$ be an increasing sequence of finite subsets, and $\Theta=(\theta_{\textbf{b}}^{(n)})_{\textbf{b}\in B_n, n\geq 0}$ be a collection of real (possibly negative) numbers. For a function $f\in L^2_0$ and a point $x\in X$ we will say that $f$ is \textit{equidistributed at $x$ along $\s{B}{n}$ with weights $\Theta$} w.r.t $\phi$ if
\begin{equation}\label{weightsparseequidef}
\frac{1}{\# B_n} \sum_{\textbf{b}\in B_n} \theta_{\textbf{b}}^{(n)} f(\phi_{\textbf{b}}(x)) \rightarrow 0 \;\;\; \as{n}.
\end{equation}

For the purposes of investigation let us fix $n$, i.e. for a finite set $B\subset \R^d$ and real numbers $\Theta=(\theta_{\textbf{b}})_{\textbf{b}\in B}$ we estimate 
\begin{equation}\label{weightsparsesumdef}
\left|\frac{1}{\# B} \sum_{\textbf{b}\in B} \theta_{\textbf{b}} f(\phi_{\textbf{b}}(x))\right|.
\end{equation}

As it turns out, the relevant object, other than twisted integrals (or sums), are the \textit{weighted exponential sums}
\begin{equation*}
S^{\Theta}_B(\xi)= \left| \sum_{\textbf{b}\in B} \theta_{\textbf{b}} e( \langle \xi, \textbf{b} \rangle) \right|.
\end{equation*}
If $\Theta\equiv 1$, then we simply write $S_B=S^{\Theta}_B$. Once we have good control on $S^{\Theta}_B$ the proof is relatively simple. Let $\chi_{\delta}$ be a bump function at $0$, say 
\begin{equation}\label{chidef}
\chi_{\delta}(x)=\delta \max(0,1-\delta^{-1} x).
\end{equation}
Then the sum in \eqref{weightsparsesumdef} is approximated by\footnote{For convenience assume $\textbf{0}\in B$.} 
\begin{equation*}
\frac{1}{\#B} \int_{[-d(B),d(B)]^d} \sum_{\textbf{b}\in B} \theta_{\textbf{b}} g_{\delta}(\textbf{t}-\textbf{b}) f(\phi_{\textbf{t}}(x) \d m(\textbf{t}),
\end{equation*}
where
\begin{equation}\label{gdef}
g(s_1,...,s_d)=\prod_{j=1}^d \chi_{\delta}(s_j),
\end{equation}
and $d(B)=\max_{\textbf{b}\in B, \textbf{b}'\in B} |\textbf{b}-\textbf{b}'|_{\infty}$. 
Deducing bounds on \eqref{weightsparsesumdef} from here is a straightforward exercise in Fourier analysis.

Convergence of the sum in \eqref{weightsparseequidef} is interesting even in the simplest case where $B_N=[-N,N]^d\cap \Z^d$ and $\theta_{\textbf{b}}^{(n)} \equiv 1$. The question then reduces down to ordinary equidistribution of the restricted action $\phi^{(\textbf{1})}$. \\
For example, by taking $f$ continuous one might investigate the basin of invariant measures of $\phi^{(\textbf{1})}$.

\begin{theorem}\label{unithm}
Let $X$ be compact. If $\phi$ is weak mixing (w.r.t $\mu$) then $ B_{\phi}(\mu)\subset B_{\phi^{(\textbf{1})}}(\mu)$. In particular, if $\phi$ is weak mixing and uniquely ergodic, then so is $\phi^{(\textbf{1})}$.\\
Furthermore, assuming $\phi$ is Hölder-continuous, and assumption (N) is satisfied if $\phi$ has polynomially decaying ergodic averages and is polynomially weak mixing, then $\phi^{(\textbf{1})}$ has polynomially decaying ergodic averages.
\end{theorem}

\begin{remark}
The assumption of compactness in Theorem \ref{unithm} can be replaced by the weaker condition\footnote{This condition is implied by continuity and compactness. Indeed, by continuity, for $\epsilon>0$ and $x\in X$, there is an open neighbourhood $U_x\subset X$ containing $x$ and a $\delta_x>0$ such that $d(y,\phi_{\textbf{t}}(y))<\epsilon$ for $|t|<\delta_x$ and $y\in U_x$. By compactness, there are finitely many $U_{x_i}$ covering $X$. It follows that $\max_{|\textbf{t}|_{\infty}<\delta, x\in X} d(x,\phi_{\textbf{t}}(x)) <\epsilon$ for $ \delta=\min(\delta_{x_i}).$}
\begin{equation}\label{contcond}
\max_{|\textbf{t}|_{\infty}<\delta, x\in X} d(x,\phi_{\textbf{t}}(x)) \rightarrow 0 \;\;\;\textit{ as } \delta\rightarrow 0.
\end{equation}
Note that \eqref{contcond} is also satisfied if $\phi$ is Hölder continuous.
\end{remark}

Note that, even if $\phi$ is not uniquely ergodic, Theorem \ref{unithm} can be used to classify the invariant measures of the restricted action $\phi^{(\textbf{1})}$ assuming that
\begin{itemize}
\item[(a)] all invariant and ergodic measures of $\phi$ are weakly mixing,
\item[(b)] $X$ is partitioned into basins of invariant and ergodic measures of $\phi$.
\end{itemize}
The problem is with the assumption (b), which is often\footnote{In fact, to the author's knowledge, (b) is only known for unipotent flows by Ratner's measure classification Theorem.} is not satisfied. On the contrary there are points $x\in X$ with \textit{mixed behaviour}, meaning there are two different invariant and ergodic measures $\mu_1$ and $\mu_2$, and two sequences $\s{T^{(1)}}{n}$ and $\s{T^{(2)}}{n}$ such that, for each bounded and continuous function $f:X\rightarrow\R$, it holds that
\begin{equation}\label{mixeddef}
\lim_{n\rightarrow\infty} \frac{1}{T^{(1)}_n} \int_0^{T^{(1)}_n} f(\phi_t(x)) \d t = \int_X f(x) \d\mu_1 ,\; \textit{ but } \; \lim_{n\rightarrow\infty} \frac{1}{T^{(2)}_n} \int_0^{T^{(2)}_n} f(\phi_t(x)) \d t = \int_X f(x) \d\mu_2.
\end{equation}
 
However, using similar techniques, we can drop assumption (b), leading us to the following.
 
\begin{theorem}\label{classthm}
Suppose $X$ is compact. If all invariant and ergodic measures of $\phi$ are weak mixing, then it holds that
\begin{align*}
\{\mu & \;|\; \mu \textit{ is an invariant probability measure for } \phi\} \\
&= \{\mu \;|\; \mu \textit{ is an invariant probability measure for } \phi^{(\textbf{1})}\}.
\end{align*}
\end{theorem}

\section{Examples}

In this paper, we shall focus on one particular examples of sparse sequences. For sparse sequences, we will focus on the sequence $(n^{1+\epsilon})$, thus generalising the result of \cite{venkateshtwist}. Although well know, this shall give a more general flair to the proof.

\begin{definition}\label{1+epsdef}
Let $x\in X$. We say that $\phi$ is \textit{pointwise equidistributed along $n^{1+\epsilon}$ at $x$} if, for small enough $\epsilon_1,...,\epsilon_d>0$ and
\begin{equation*}
B_N=\{ (n_1^{1+\epsilon_1},...,n_d^{1+\epsilon_d}) | n_j\in [0,N-1]^d\cap \Z^d \forall j=1,..,d\},
\end{equation*}
we have
\begin{equation*}
\frac{1}{N^d} \left|\sum_{\textbf{b}\in B_N} f(\phi_{\textbf{b}}(x)) \right| \rightarrow 0 \;\;\; \as{N}, \; \forall f\in C^0\cap L^2_0.
\end{equation*}
In this case, the rates are said to decay polynomially (in $\B$) if there is some $\delta>0$ such that
\begin{equation*}
\frac{1}{N^d} \left|\sum_{\textbf{b}\in B_N} f(\phi_{\textbf{b}}(x)) \right| \ll ||f||_{\B} N^{-\delta} \;\;\; \forall N \geq 1, \forall f\in \B\cap L^2_0.
\end{equation*}
\end{definition}

\begin{theorem}\label{explethm}
 Assume $\phi$ is Hölder-continuous and assumption (N) is satisfied. Let $x\in X$. Suppose $\phi$ has polynomially decaying twisted integrals at $x$, then $\phi$ is pointwise equidistributed along $n^{1+\epsilon}$ at $x$, the rates decay polynomially in $\B$.
\end{theorem}

\begin{remark}\label{linrem}
Let us remark at this point that linear flows $R_{\alpha}$, even though they do not fit into the setting of Theorem\footnote{They are not weakly mixing.} \ref{explethm}, are equidistributed along $n^{1+\epsilon}$ for all $\epsilon>0$. \\
Indeed, \cite[Theorem 3.5]{kuipers1974uniform} implies that, for every $\xi\in \R\setminus \{0\}$ and $p>0$ (note that there is no upper bound on $p$) it holds that 
\begin{equation*}
\frac{1}{N} \sum_{n=0}^{N-1} e(\xi n^p) \rightarrow 0 \;\;\;\as{N}.
\end{equation*}

The arguments in \cite{kuipers1974uniform} can easily be made quantitative\footnote{Which shall be carried out in \S \ref{linsec}.}, to show that, if the angle is diophantine, the rate of convergence is polynomial.
\end{remark}

Below, we present a list of flows, to which one can apply Theorem \ref{explethm} at \textbf{all} points $x\in X$. 
We do not make any claims on completeness, in fact, it is to be expected that Theorem \ref{explethm} can be applied to most polynomially mixing parabolic flows. In the applications below, we take $\B=C^r$ for some $r\geq 1$. All the flows appearing on our list shall be defined precisely in \S \ref{explesec}.

\begin{enumerate}[(a)]
\item a smooth time changes of unipotent flow on compact space
(In this case polynomial mixing \textbf{always} holds. 
The only restriction to applying our results is polynomial 
equidistribution. This requires
quantitative versions of Ratner's Theorem, which is highly non-trivial and an active field of study.);

\item a generic smooth time change of a Heisenberg nilflow satisfying a diophantine condition.
\end{enumerate}

On the side of higher dimensional actions, not that many examples have been studied. In \S \ref{ddimhorosec} we shall describe, in analogue to the horocyle flow, an $\R^d$ action on a compact homogeneous space, to which we can apply Theorem \ref{explethm}. In particular, regarding equidistribution along\footnote{Rather, on sets of the form $\{(n_1^{1+\epsilon_1},...,n_d^{1+\epsilon_d})\}$ as in Definition \ref{1+epsdef}.} $n^{1+\epsilon}$ this represents a higher dimensional analogue of Venkatesh's result (which is the case of $d=1$).

\part{Proofs}
\section{Weak Mixing and $L^2$ Twisted Integrals}\label{mixtwistsec}

This section will be devoted to proving Theorem \ref{wmtwistthm}.

\begin{lemma}\label{l2intineqlem}
For $f\in L^{2}([-T,T]^d\times well-defined\cdot,x)\in L^1([-T,T]^d)$ for a.e $x$ we have
\begin{equation}\label{ineql2int}
\left|\left| \int_{[-T,T]^d} f(\textbf{t},x) \d m(\textbf{t}) \right|\right|_{L^2(X)}\leq  \int_{[-T,T]^d} \left|\left| f(\textbf{t},x) \right|\right|_{L^2(X)} \d m(\textbf{t}),
\end{equation}
whenever both sides are well-defined and finite.\\
If $f\in L^{2}(\R^d\times X)$ with $f(\cdot,x)\in L^1(\R^d)$ for a.e $x$ we have
\begin{equation}\label{ineql2intrd}
\left|\left| \int_{\R^d} f(\textbf{t},x) \d m(\textbf{t}) \right|\right|_{L^2(X))}\leq  \int_{\R^d} \left|\left| f(\textbf{t},x) \right|\right|_{L^2(X)} \d m(\textbf{t}),
\end{equation}
whenever both sides are well-defined and finite.
\end{lemma}
\begin{proof}
\begin{enumerate}[(i)]
\item First, assume that $f$ is the product of two measurable functions, i.e. there are measurable $\phi\in L^2(X), \psi\in L^1([-T,T]^d)$ such that $f(\textbf{t},x)=\phi(x)\psi(\textbf{t})$. Then
\begin{align*}
&\left|\left| \int_{[-T,T]^d} \phi(x)\psi(\textbf{t}) \d m(\textbf{t}) \right|\right|_{L^2(X)}=||\phi||_{L^2(X)} \left|\int_{[-T,T]^d} \psi(\textbf{t}) \d m(\textbf{t})\right|, \; \textit{ and}\\
&\int_{[-T,T]^d} \left|\left| f(t,x) \right|\right|_{L^2(X)} \d m(\textbf{t})=||\phi||_{L^2(X)} \int_{[-T,T]^d} |\psi(\textbf{t})| \d m(\textbf{t}). 
\end{align*}
\item Suppose inequality \eqref{ineql2int} holds for $f_1,f_2\in L^2([-T,T]\times X)$, and $\supp(f_1)\cap \supp(f_2)=\emptyset$. Then
\begin{align*}
&\left|\left| \int_{[-T,T]^d} f_1(\textbf{t},x)+f_2(\textbf{t},x) \d m(\textbf{t}) \right|\right|_{L^2(X)} \\
&\leq \left|\left| \int_{[-T,T]^d} f_1(\textbf{t},x) \d m(\textbf{t}) \right|\right|_{L^2(X)}+\left|\left| \int_{[-T,T]^d} f_2(\textbf{t},x) \d m(\textbf{t}) \right|\right|_{L^2(X)} \\
&\leq \int_{[-T,T]^d} \left|\left| f_1(\textbf{t},x) \right|\right|_{L^2(X)} \d m(\textbf{t}) +\int_{[-T,T]^d} \left|\left| f_2(\textbf{t},x) \right|\right|_{L^2(X)} \d m(\textbf{t}) \\
&\leq \int_{[-T,T]^d} \left|\left| f_1(\textbf{t},x) \right|\right|_{L^2(X)} + \left|\left| f_2(\textbf{t},x) \right|\right|_{L^2(X)} \d m(\textbf{t}) \\
&\leq \int_{[-T,T]^d} \left|\left| f_1(\textbf{t},x)+f_2(\textbf{t},x) \right|\right|_{L^2(X)} \d m(\textbf{t}). 
\end{align*} 

So the inequality \eqref{ineql2int} also holds for $f_1+f_2$.
\item Suppose that inequality \eqref{ineql2int} holds for $f_n\in L^2$ for $n\geq 1$, and assume that there is a $f\in L^2$ such that $f_n \nearrow f$ pointwise. Due to the Monotone Convergence Theorem, we have
\begin{align*}
&\left|\left| \int_{[-T,T]^d} f_n(\textbf{t},x) \d m(\textbf{t}) \right|\right|_{L^2(X)} \nearrow \left|\left| \int_{[-T,T]^d} f(\textbf{t},x) \d m(\textbf{t}) \right|\right|_{L^2(X)},\; \textit{ and}\\
&\int_{[-T,T]^d} \left|\left| f_n(\textbf{t},x) \right|\right|_{L^2(X)} \d m(\textbf{t}) \nearrow \int_{[-T,T]^d} \left|\left| f(\textbf{t},x) \right|\right|_{L^2(X)} \d m(\textbf{t}).
\end{align*}
Thus inequality \eqref{ineql2int} also holds for $f$.
\item For arbitrary $f\in L^2$ it is standard to find $\phi_{n,k},\psi_{n,k}$ such that $\supp \phi_{n,k} \cap \supp \phi_{n,k'} = \emptyset$ and $\supp \psi_{n,k} \cap \supp \psi_{n,k'} = \emptyset$, $\sum_k \phi_{n,k}\psi_{n,k}\in L^2$ and $\sum_k \phi_{n,k}\psi_{n,k} \nearrow f$. Now apply (i)-(iii).
\end{enumerate}

The second claim follows easily from the first; we have
\begin{equation*}
\left|\left| \int_{[-T,T]^d} f(\textbf{t},x) \d m(\textbf{t}) \right|\right|_{L^2(X)}\leq  \int_{[-T,T]^d} \left|\left| f(\textbf{t},x) \right|\right|_{L^2(X)} \d m(\textbf{t}),
\end{equation*}
for every $T>0$. Now the claim follows from $T\rightarrow\infty$ and dominated convergence.
\end{proof}

\begin{proposition}\label{wmtwistprop1}
For all $f\in L^2_0$, $\textbf{a}\in \R^d$ or $[0,1)^d$), $x\in X$ and $T,H>0$ we have
\begin{equation}\label{wmtwistprop1rates}
\frac{1}{(2T)^d}\left|\left| \int_{[-T,T]^d} f(\phi_{\textbf{t}} x) e( \langle \textbf{a} , \textbf{t}\rangle) \d m(\textbf{t}) \right|\right|_{L^2(X)} 
\ll \alpha_{f,f}(T)^{\frac{1}{2}} .
\end{equation}
\end{proposition}
\begin{proof}
We have
\begin{align*}
&\left|\left| \int_{[-T,T]^d} f(\phi_{\textbf{t}} x) e( \langle \textbf{a} , \textbf{t}\rangle) \d m(\textbf{t}) \right|\right|_{L^2(X)}^2 \\
&=\int_X \int_{[-T,T]^d} \int_{[-T,T]^d} f(\phi_{\textbf{t}_1} x) \overline{f(\phi_{\textbf{t}_2} x)} e( \langle \textbf{a}, \textbf{t}_1-\textbf{t}_2 \rangle) \d m(\textbf{t}_1) \d m(\textbf{t}_2) \d \mu(x)\\
&\ll \int_{[-T,T]^d} \int_{[-T,T]^d} | \langle f(\phi_{\textbf{t}_1}), f(\phi_{\textbf{t}_2})\rangle | \d m(\textbf{t}_1) \d m(\textbf{t}_2)\\
&\ll \int_{[-T,T]^d} \int_{[-T,T]^d} | \langle f(\phi_{\textbf{t}_1-\textbf{t}_2}), f\rangle | \d m(\textbf{t}_1) \d m(\textbf{t}_2)\\
&\ll T^d \int_{[-T,T]^d} | \langle f(\phi_{\textbf{u}}), f\rangle | \d m(\textbf{u})\\
&\ll T^{2d} \alpha_{f,f}(T).
\end{align*}  

\end{proof}

Proposition \ref{wmtwistprop1} says that twisted integrals decay whenever $\phi$ is weakly mixing. In fact, the converse also is true. In order to show this, we will take a detour via the theory of spectral measures.

For $f\in L^2(X)$ we will call a measure $\sigma_f$ on $\R^d$ (or $\T^d$) the \textit{spectral measure} of $f$ w.r.t $\phi$ if
\begin{equation*}
\langle f(\phi_{\textbf{t}}),f\rangle = \hat{\sigma}_f(\textbf{t}) = \int_{\hat{G}} e( \langle \textbf{a}, \textbf{t}\rangle) \d\sigma_f(\textbf{a}) \;\;\;\forall \textbf{t}\in G.
\end{equation*}
Such measures, by the Bochner-Herglotz Theorem, always exist and are finite, in fact, $\sigma_f(G)=||f||_{L^2(X)}$. A proof of existence can be found in most books on Ergodic Theory, see for example \cite{KATOKspectral} or \cite{parry2004topics}.\\
For ease of notation, in the case $\hat{G}=\T^d$, we identify $\sigma_f$ with a measure on $\left[0,1\right)^d$ extended periodically to $\R^d$.\\
The following Lemma is a straightforward higher dimensional generalisation (using the same method of proof) of Lemma 5.3 of \cite{quantwmafs}.

\begin{lemma}\label{twistspeclem}
Let $\sigma_f$ be the spectral measure of $f\in L^2_0$, then
\begin{equation*}
\sigma_f\left( -\textbf{a}+ \left[-\frac{1}{T},\frac{1}{T}\right]^d \right) \ll T^{-2d} \left|\left| \int_{[-T,T]^d} f(\phi_{\textbf{t}}) e( \langle\textbf{a}, \textbf{t}\rangle ) \d m(\textbf{t}) \right|\right|_{L^2(X)}^2,
\end{equation*}
for $T>1$ and $\textbf{a}\in \hat{G}$.
\end{lemma}
\begin{proof}
First we rewrite
\begin{align*}
&\left|\left| \int_{[-T,T]^d} f(\phi_{\textbf{t}}) e( \langle\textbf{a}, \textbf{t}\rangle) \d m(\textbf{t})\right|\right|_{L^2(X)}^2 \\
& = \int_X \int_{[-T,T]^d} \int_{[-T,T]^d} f(\phi_{\textbf{t}_1}) \overline{f(\phi_{\textbf{t}_2})} e( \langle \textbf{a}, \textbf{t}_1-\textbf{t}_2 \rangle) \d m(\textbf{t}_1) \d m(\textbf{t}_2) \d\mu\\
&=\int_{[-T,T]^d} \int_{[-T,T]^d} e( \langle \textbf{a}, \textbf{t}_1-\textbf{t}_2 \rangle) \langle f(\phi_{\textbf{t}_1}), f(\phi_{\textbf{t}_2}) \rangle\d m(\textbf{t}_1) \d m(\textbf{t}_2)\\
&=\int_{[-T,T]^d} \int_{[-T,T]^d} \int_{\hat{G}} e( \langle \textbf{b} + \textbf{a}, \textbf{t}_1-\textbf{t}_2 \rangle) \d\sigma_f(\textbf{b})\d m(\textbf{t}_1) \d m(\textbf{t}_2)\\
&=\int_{\hat{G}} \left|\int_{[-T,T]^d} e( \langle \textbf{b} + \textbf{a}, \textbf{t} \rangle) \right|^2 \d\sigma_f(\textbf{b}).
\end{align*}
We claim 
\begin{equation}\label{twistspeclemclaim}
T^{2d}\ll \left|\int_{[-T,T]^d} e( \langle \textbf{b} + \textbf{a}, \textbf{t} \rangle) \d m(\textbf{t})\right|^2 ,
\end{equation}
whenever $|\textbf{a}+\textbf{b}|_{\infty}\leq \frac{1}{16\pi T}$.
Indeed, if $G=\R^d$, then 
\begin{equation*}
\left|\int_{[-T,T]^d} e( \langle \textbf{b} + \textbf{a}, \textbf{t} \rangle) \right|^2=\prod_{j=i}^d  \left| \frac{e((a_j+b_j)T)-1}{2\pi (a_j+b_j)}\right|^2.
\end{equation*}
If $|a_j+b_j|T<\frac{1}{16\pi}$, where $\textbf{a}=(a_1,...,a_d)$ and $\textbf{b}=(b_1,...,b_d)$, then $\left|e( (a_j+b_j)T)-1\right|\sim 2\pi |a_j+b_j|T$, so \eqref{twistspeclemclaim} follows.

If $G=\Z^d$ we may assume $T\in \mathbb{N}$. Then
\begin{equation*}
\left|\int_{[-T,T]^d} e( \langle \textbf{b} + \textbf{a}, \textbf{t} \rangle) \right|^2=\prod_{j=i}^d  \left| \frac{e( (a_j+b_j)T)-1}{e( (a_j+b_j) -1}\right|^2
\end{equation*}
 and \eqref{twistspeclemclaim} follows by the same argument.
By \eqref{twistspeclemclaim},
\begin{align*}
T^{2d} & \sigma_f\left( -\textbf{a}+ \left[-\frac{1}{16\pi T},\frac{1}{16\pi T}\right]^d \right)\\
&\ll\int_{-\textbf{a}+ \left[-\frac{1}{16\pi T},\frac{1}{16\pi T}\right]} \left|\int_{[-T,T]^d} e( \langle \textbf{b} + \textbf{a}, \textbf{t} \rangle) \right|^2 \d\sigma_f(\textbf{b})\\
&\ll\left|\left| \int_{[-T,T]^d} f(\phi_{\textbf{t}}) e( \langle\textbf{a}, \textbf{t}\rangle) \d m(\textbf{t}) \right|\right|_{L^2(X)}^2.
\end{align*}
Replacing $\frac{1}{16\pi T}$ by $\frac{1}{T}$ concludes the proof.
\end{proof}

Concluding weak mixing from this 'decay' of spectral measures is standard and can be found in most books on ergodic theory. In an effort to keep these notes self-contained, and to highlight the rates of decay, we will present a proof. Here we follow the arguments of \cite{einsiedler2010ergodic}.

\begin{lemma}\label{expsumlem}
Let $\textbf{a}=(a^{(1)},...,a^{(d)}) \in \hat{G}$, then
\begin{equation*}
\left|\int_{[-T,T]^d} e( \langle \textbf{a}, \textbf{t} \rangle) \d m(\textbf{t})\right| \ll \prod_{j=1}^d \min(T,|a^{(j)}|^{-1}).
\end{equation*}
\end{lemma}

\begin{lemma}\label{specwmlem}
Denote
\begin{equation*}
\sup_{\textbf{a}\in \hat{G}} \sigma_f\left( -\textbf{a}+ \left[-\frac{1}{T},\frac{1}{T}\right]^d \right) =\delta_f(T),
\end{equation*}
Then, for $\epsilon>0$, we have
\begin{equation*}
\int_{[-T,T]^d} \left|\langle f(\phi_{\textbf{t}}, f\rangle \right| \d m(\textbf{t}) \ll (2T)^d\delta_f \left(\frac{1}{\epsilon}\right)^{\frac{1}{2}} + ||f||_{L^2(X)} T^{d-\frac{1}{2}}\epsilon^{-\frac{1}{2}}.
\end{equation*}
\end{lemma}
\begin{proof}
By Cauchy-Schwartz we have
\begin{equation*}
\int_{[-T,T]^d} \left|\langle f(\phi_{\textbf{t}}, f\rangle \right| \d m(\textbf{t}) \leq T^{\frac{d}{2}} \left( \int_{[-T,T]^d} \left|\langle f(\phi_{\textbf{t}}, f\rangle \right|^2 \d m(\textbf{t}) \right)^{\frac{1}{2}} 
\end{equation*}
so it suffices to bound $\int_{[-T,T]^d} \left|\langle f(\phi_{\textbf{t}}, f\rangle \right|^2 \d m(\textbf{t})$. We have
\begin{align*}
\int_{[-T,T]^d} \left|\langle f(\phi_{\textbf{t}}, f\rangle \right|^2 \d m(\textbf{t}) &= \int_{[-T,T]^d} \hat{\sigma_f}(\textbf{t}) \hat{\sigma_f}(-\textbf{t}) \d m(\textbf{t})\\
&= \int_{\hat{G}} \int_{\hat{G}} \int_{[-T,T]^d} e( \langle \textbf{a}_1-\textbf{a}_2 , \textbf{t} \rangle) \d\sigma_f(\textbf{a}_1) \d\sigma_f(\textbf{a}_2) \d m(\textbf{t}). 
\end{align*}
Denote $\psi_T(\textbf{a}_1,\textbf{a}_2)=\left| \int_{[-T,T]^d} e( \langle \textbf{a}_1-\textbf{a}_2 , \textbf{t} \rangle) \d m(\textbf{t}) \right|$. For $\epsilon>0$ we claim
\begin{equation}\label{specwmlemclaim}
\int_{|\textbf{a}_1-\textbf{a}_2|_{\infty} > \epsilon} \psi_t(\textbf{a}_1,\textbf{a}_2)  \d\sigma_f(\textbf{a}_1) \d\sigma_f(\textbf{a}_2) \ll \epsilon^{-1} T^{d-1}.
\end{equation}
For a non-empty $\mathcal{J}\subset \{1,...,d\}$ let
\begin{equation*}
C_{\mathcal{J}}=\{(\textbf{a}_1,\textbf{a}_2)\in \hat{G}\times \hat{G}\;|\; |a_1^{(j)}-a_2^{(j)}|> \epsilon \;\iff \;j\in \mathcal{J}\}.
\end{equation*}
Then by Lemma \ref{expsumlem} we have $\psi_T (\textbf{a}_1,\textbf{a}_2)< \epsilon^{-\#\mathcal{J}} T^{d-\#\mathcal{J}}$, furthermore
\begin{equation*}
\{|\textbf{a}_1-\textbf{a}_2|_{\infty} > \epsilon\} = \bigcup_{\emptyset\not= \mathcal{J} \subset \{1,...,d\}} C_{\textbf{J}}.
\end{equation*}
The claim \eqref{specwmlemclaim} follows.

So
\begin{align*}
&\int_{[-T,T]^d} \left|\langle f(\phi_{\textbf{t}}, f\rangle \right|^2 \d m(\textbf{t}) \ll \int_{\hat{G}} \int_{\hat{G}} \psi_T(\textbf{a}_1, \textbf{a}_2) \d\sigma_f(\textbf{a}_1) \d\sigma_f(\textbf{a}_2) \\
&\ll \int_{|\textbf{a}_1-\textbf{a}_2|_{\infty} \leq \epsilon} (2T)^d \d\sigma_f(\textbf{a}_1) \d\sigma_f(\textbf{a}_2) + \int_{|\textbf{a}_1-\textbf{a}_2|_{\infty} > \epsilon} \psi_t(\textbf{a}_1,\textbf{a}_2)  \d\sigma_f(\textbf{a}_1) \d\sigma_f(\textbf{a}_2)\\
&\ll \int_{\hat{G}} (2T)^d \sigma_f\left( -\textbf{a}+ \left[-\epsilon, \epsilon\right]^d \right) \d\sigma_f(\textbf{a}) + ||f||_{L^2(X)}^2 \epsilon^{-1} T^{d-1}\\
&\ll (2T)^d\delta_f \left(\frac{1}{\epsilon}\right) + ||f||_{L^2(X)}^2 \epsilon^{-1} T^{d-1}.
\end{align*}

\end{proof}

\begin{proof}[Proof of Theorem \ref{wmtwistthm}]
(i) If $\phi$ is weak mixing (or polynomially weak mixing), then Proposition \ref{wmtwistprop1} immediately yields the desired decay of twisted integrals.

(ii) On the other hand assume that twisted integrals decay. By Lemma~\ref{twistspeclem}, for $f\in L^2_0$ (or $f\in \B\cap L^2_0$), we have
\begin{equation*}
\sigma_f\left( -\textbf{a}+ \left[-\frac{1}{T},\frac{1}{T}\right]^d \right) \ll T^{-2d} \left|\left| \int_{[-T,T]^d} f(\phi_{\textbf{t}}) e( \langle\textbf{a}, \textbf{t}\rangle) \d m(\textbf{t}) \right|\right|_{L^2(X)}^2\ll \gamma_f(T)^2.
\end{equation*}
Lemma \ref{specwmlem} yields
\begin{equation}\label{acineq}
\alpha_{f,f}(T) \ll \sigma_f\left( -\textbf{a}+ \left[-\epsilon,\epsilon \right]^d \right) + ||f||_{L^2(X)} T^{-\frac{1}{2}}\epsilon^{-\frac{1}{2}}\ll \gamma_f(\epsilon^{-1}) + ||f||_{L^2(X)} T^{-\frac{1}{2}}\epsilon^{-\frac{1}{2}},
\end{equation}
for all $\epsilon>0$. Choosing $\epsilon=T^{-\frac{1}{2}}$ yields $\alpha_{f,f}(T) \rightarrow 0$ as $T\rightarrow \infty$. \\
For $f,g\in L^2_0$ we have
\begin{equation}\label{poleq}
\alpha_{f,g}(T) \leq \frac{1}{4}(\alpha_{f+g,f+g}(T)+\alpha_{f-g,f-g}(T)),
\end{equation} 
hence also $\alpha_{f,g}(T)\rightarrow 0$ as $T\rightarrow\infty$.

If twisted integrals decay polynomially, say $\gamma_f(T)\ll ||f||_\B T^{-\kappa}$, then setting $\epsilon=T^{-\min(1,\kappa^{-1})}$ in \eqref{acineq} yields
\begin{equation*}
\alpha_{f,f}(T) \ll ||f||_{\B} T^{-\min\left(1,\kappa\right)}.
\end{equation*}
The polynomial bound on $\alpha_{f,g}(T)$ for $f\neq g$ also follows from \eqref{poleq}.
\end{proof}

\begin{remark}
In Proposition \ref{wmtwistprop1} and Lemmas \ref{twistspeclem} and \ref{specwmlem} we can easily replace $[-T, T]^d$ by another increasing sequence of sets $A_T\subset G$ satisfying\footnote{In particular if $A_T$ is a F$\o$lner sequence, then \eqref{poleq} holds.} 
\begin{equation}\label{wmremeq1}
\int_{A_T} e( \langle \textbf{a}, \textbf{t}\rangle) \d m(\textbf{t}) =o(m(A_T)) \;\;\;\as{T},
\end{equation}
for all $\textbf{a}\not=0 \in \hat{G}$.\footnote{We implicitly assume $m(A_T)\rightarrow \infty$ otherwise the statement below makes no sense.} For such a sequence $A_T$, we have
\begin{equation}\label{wmremclaim1}
\int_{A_T} |\langle f(\phi_{\textbf{t}}),f\rangle | \d m(\textbf{t}) \rightarrow 0 \;\;\; \as{T},
\end{equation}
if $\phi$ is weak mixing. The rate of convergence depends (in a way easily seen from the above lemmas) on the speed of weak mixing and the growth rate in \eqref{wmremeq1}.
\end{remark}

\section{The restricted action} \label{restrictsec}

For a function $h\in L^2$, denote its \textit{Fourier transform} by
\begin{equation*}
\F(h)(\xi)=\int_{\R^d} e(- \langle \xi,\textbf{x}\rangle) h(\textbf{x}) \d \textbf{x} \;\;\; \xi\in \R^d.
\end{equation*}

For a function $h\in C^{\infty}_c$ and $A>0$ such that $\supp h\subset [-A,A]^d$ denote the \textit{Fourier series} by
\begin{equation*}
\F_A(h)(\textbf{k})=\frac{1}{A^d} \int_{[-A,A]^d} e(-\frac{1}{A} \langle \textbf{k},\textbf{x}\rangle) h(\textbf{x}) \d \textbf{x} \;\;\; \textbf{k}\in \Z^d.
\end{equation*}
Standard theory, for $h\in L^2$, shows that $h(\textbf{x})= \int_{\R^d} \F(h)(\xi) e( \langle \xi,\textbf{x}\rangle) \d\xi$, and $h(\textbf{x})=\sum_{\textbf{k}\in \Z^d} \F_A(h)(\textbf{k}) e\left(\frac{1}{A} \langle \textbf{k},\textbf{x}\rangle\right)$ for $\textbf{x}\in [-A,A]^d$.\\
In the following, we prove equidistribution (and more) along a subsequence. 
In the proofs one may use $\F$ or $\F_A$, which one we use does not make much of a difference, but one might be more convenient than the other (we will remark on this later).

\begin{lemma}\label{l2twistl1lem}
For $f\in L^2_0$, $T>0$ and $\psi\in L^2$ with $\F(\psi)\in L^1$, we have
\begin{equation}\label{l2twistlemeq1}
\frac{1}{(2T)^d} \left| \left| \int_{[-T,T]^d} f(\phi_{\textbf{t}}(x)) \psi(\textbf{t}) \d m(\textbf{t}) \right|\right|_{L^2(X)} \ll ||\F(\psi)||_{L^1(\R^d)} \gamma_f(T) .
\end{equation}
Furthermore if $\supp \psi \subset [-A,A]^d$ and $\F_A(\psi)\in l^1$, then
\begin{equation*}
\frac{1}{(2T)^d} \left| \left| \int_{[-T,T]^d} f(\phi_{\textbf{t}}(x)) \psi(\textbf{t}) \d m(\textbf{t}) \right|\right|_{L^2(X)} \ll ||\F_A(\psi)||_{l^1} \gamma_f(T) .
\end{equation*}
\end{lemma}
\begin{proof}
We only prove \eqref{l2twistlemeq1}, the other assertion can be shown analogously. By Fubini's Theorem and Lemma \ref{l2intineqlem} we have
\begin{align*}
&\left| \left| \int_{[-T,T]^d} f(\phi_{\textbf{t}}(x)) \psi(\textbf{t}) \d m(\textbf{t}) \right|\right|_{L^2(X)}\\
&\leq \left| \left| \int_{[-T,T]^d} \int_{\R^d} \F(\psi)(\xi) e( \langle \xi, \textbf{t} \rangle) \d\xi f(\phi_{\textbf{t}}(x)) \d m(\textbf{t}) \right|\right|_{L^2(X)}\\
&\leq \int_{\R^d} |\F(\psi)(\xi)| \d\xi \left| \left| \int_{[-T,T]^d} e( \langle \xi, \textbf{t} \rangle) f(\phi_{\textbf{t}}(x)) \d m(\textbf{t}) \right|\right|_{L^2(X)}.
\end{align*}

\end{proof}

To show decay of weighted sums along sparse sequences, say a sequence of finite $B_N\subset \R^d$ (or $\Z^d$) and weights $\Theta=(\theta_{\textbf{b}}^{(n)})$, we will use the following strategy. For $\delta>0$ let 
\begin{equation*}
\chi_{\delta}(t)= \delta^{-1} \max(0, 1-\delta^{-1} t), \; g_{\delta}(\textbf{t})=\prod_{j=1}^d \chi_{\delta}(t_j),
\end{equation*}
where $\textbf{t}=(t_1,...,t_d)$, and
\begin{equation}\label{Gdef}
G_{\delta,N}(\textbf{t})=\sum_{\textbf{b}\in B_N} \theta_{\textbf{b}}^{(N)} g_{\delta}(\textbf{t}-\textbf{b}),
\end{equation}
where $\textbf{t}=(t_1,...,t_d)$.\\
Since $g$ is a bump function at $\textbf{0}$, heuristically we have $G_{\delta,N} \sim \sum_{\textbf{b}\in B_N} \theta_{\textbf{b}}^{(N)} \delta_{\textbf{b}}$
for small $\delta$.
At the same time, by Lemma \ref{l2twistl1lem}
\begin{equation*}
\int_{[-T,T]^d} G_{\delta,N}(\textbf{t}) e( \langle \textbf{a}, \textbf{t}\rangle) f(\phi_{\textbf{t}}(x)) \d \textbf{t}=o(||\hat{G}_{\delta,N}||_{L^1(\R^d)}),
\end{equation*}
if $\phi$ is weakly mixing.
The latter term depends primarily on ($\delta$ and) the exponential sums 
\begin{equation*}
\left|\sum_{\textbf{b}\in B_N} \theta_{\textbf{b}}^{(N)} e( \langle \textbf{b},\textbf{x} \rangle)\right|,
\end{equation*}
which for most $\textbf{x}$ (choosing carefully $B_N$ and $\Theta$) is of order $o(\# B_N)$. This will be made precise in the next Lemmas.

\begin{lemma}\label{fouriercoeffchilem}
For $\delta>0$, we have 
\begin{equation*}
\F(\chi_{\delta})(0)=1 \;\textit{ and }\; \F(\chi_{\delta})(\xi)=\frac{1}{2\xi^2\pi^2\delta^2} \left(1-\cos \left( 2\pi\xi\delta\right)\right),
\end{equation*}
for $\xi\not=0$. For $A>\delta>0$, we have 
\begin{equation*}
\F_A(\chi_{\delta})(0)=A^{-1} \;\textit{ and }\; \F_A(\chi_{\delta})(k)=\frac{A}{2k^2\pi^2\delta^2} \left(1-\cos \left( \frac{2\pi}{A} k\delta\right)\right),
\end{equation*}
for $k\not=0$. Furthermore $||\F(\chi_{\delta})||_{L^1}\ll \delta^{-1}$ and $||\F_A(\chi_{\delta})||_{l^1}\ll A^{-1}+\delta^{-1}$.
\end{lemma}
\begin{proof}
We prove the claims only for $\F$, the other assertion can be shown analogously. For $\xi\not=0$ we compute
\begin{align*}
\F(\chi_{\delta})(\xi)&=  \int_{-\delta}^{\delta} e^{2\pi i \xi x} \chi_{\delta}(x) \d x\\
&= \frac{2}{\delta} \int_0^{\delta} \cos \left( 2\pi \xi x\right) (1-\delta^{-1} x) \d x\\
&= \frac{1}{\xi \pi\delta} \left[ \sin \left( 2\pi \xi x\right) (1-\delta^{-1} x) \right]_0^{\delta} + \frac{1}{\xi\pi\delta^2} \int_0^{\delta} \sin \left( 2\pi \xi x\right) \d x\\
&= -\frac{1}{2\xi^2\pi^2\delta^2} \left[\cos \left( 2\pi \xi x\right)\right]_0^{\delta}. 
\end{align*}

The Intermediate Value Theorem yields
\begin{equation*}
\left|1-\cos \left( 2\pi \xi \delta\right)\right|\ll \zeta^2,
\end{equation*}
for some $\zeta\in \left[0, 2\pi \xi\delta \right]$. We get
\begin{equation*}
\int_{\R} |\F(\chi_{\delta})(\xi)| \d \xi \ll \int_{|\xi|\leq \delta^{-1}} \frac{1}{\xi^2 \delta^2} \left|1-\cos \left( 2\pi \xi\delta\right) \right| +\int_{|\xi|\geq \delta^{-1}} \frac{1}{\xi^2 \delta^2} 
\ll \delta^{-1}.
\end{equation*}

\end{proof}

For a finite set $B\subset R^d$ and weights $\Theta=(\theta_{\textbf{b}})$, recall the notation for exponential sums
\begin{equation*}
S_{B}(\textbf{x})=\left|\sum_{\textbf{b}\in B_N} e( \langle \textbf{b},\textbf{x} \rangle)\right|,
\end{equation*}
for $\textbf{x}\in \R^d$.\\
Also, we denote $d(B)=\max(|\textbf{b}|_{\infty} \;|\; \textbf{b}\in B)$.

\begin{proposition}\label{sparseequilem}
Let $B\subset \R^d$ be a finite set and $\Theta=(\theta_{\textbf{b}})$ be real (possibly negative) numbers. If $\phi$ has decaying twisted integrals in $L^2$ of rate $\gamma$, then, for each $f\in L^2_0\cap \B$ and $\textbf{a}\in \R^d$, 
\begin{equation}\label{sparseequilem1crates}
\begin{aligned}
&\frac{1}{\# B}\left|\left|\sum_{\textbf{b}\in B} \theta_{\textbf{b}} f(\phi_{\textbf{b}}(x)) e( \langle \textbf{a},\textbf{b}\rangle) \right|\right|_{L^2(\d\mu(x)}\\
& \ll \frac{1}{\# B}+  \max_{|\textbf{t}|_{\infty}<\delta, x\in X} |f(x)-f(\phi_{\textbf{t}}(x))| \max_{\textbf{b}\in B}|\theta_{\textbf{b}}| \\&
+ \frac{(d(B))^d}{\#B} \gamma_f(d(B)) \int_{\R^d} S_{B}^{\Theta} \left(\xi \right) \prod_{j=1}^d |\F(\chi_{\delta})(\xi_j)| \d\xi.
\end{aligned}
\end{equation}
and
\begin{equation}\label{sparseequilem1rates}
\begin{aligned}
&\frac{1}{\# B}\left| \left| \sum_{\textbf{b}\in B} \theta_{\textbf{b}} f(\phi_{\textbf{b}}(x)) e( \langle \textbf{a},\textbf{b}\rangle) \right| \right|_{L^2(X)}\\
& \ll \frac{1}{\# B}+  \max_{|\textbf{t}|_{\infty}<\delta, x\in X} |f(x)-f(\phi_{\textbf{t}}(x))| \max_{\textbf{b}\in B}|\theta_{\textbf{b}}|\\
&\;\;\; + \frac{d(B)^d}{\#B} \gamma_f(d(B)) \sum_{\textbf{k}=(k_1,...,k_d)\in \Z^d} S_{B}^{\Theta} \left(\frac{1}{d(B)}\textbf{k}\right) \prod_{j=1}^d |\F_{d(B)}(\chi_{\delta})(k_j)|.
\end{aligned}
\end{equation}
for all $\delta>0$.
\end{proposition}
\begin{proof}
We show only \eqref{sparseequilem1crates}, the other assertion can be shown analogously. For simplicity assume $\textbf{0}\in B$, otherwise the proof needs only minimal adjustment. Let 
\begin{equation*}
G_{\delta}(\textbf{t})=\sum_{\textbf{b}\in B} \theta_{\textbf{b}} g_{\delta}(\textbf{t}-\textbf{b}),
\end{equation*}
then
\begin{equation}\label{wmbncGerror}
\begin{aligned}
&\left| \int_{\tilde{B}} G_{\delta}(\textbf{t}) e( \langle \textbf{a}, \textbf{t}\rangle) f(\phi_{\textbf{t}}(x)) \d \textbf{t} - \sum_{\textbf{b}\in B} \theta_{\textbf{b}} f(\phi_{\textbf{b}}(x)) e( \langle \textbf{a},\textbf{b}\rangle) \right|\\
&\ll 1+ \max_{|\textbf{t}|_{\infty}<\delta, x\in X} |f(x)-f(\phi_{\textbf{t}}(x))| \#B \max_{\textbf{b}\in B}|\theta_{\textbf{b}}|,
\end{aligned}
\end{equation}
where\footnote{Since $\textbf{0}\in B$, we have $B\subset \tilde{B}$. If we drop the assumption $\textbf{0}\in B$, $B$ is still contained in a translate of $\tilde{B}$. In this case we only have to shift the phase, and essentially the same proof is still valid.} $\tilde{B}=[-d(B),d(B)]^d$.

We have $\F(G_{\delta})(\xi)=\sum_{\textbf{b}\in B} \theta_{\textbf{b}} e(- \langle \xi,\textbf{b}\rangle) \F(g_{\delta})(\xi)$, therefore
\begin{equation}\label{fouriercGeq1}
||\F(G_{\delta})||_{L^1} = \int_{\R^d} \left|\sum_{\textbf{b}\in B} \theta_{\textbf{b}} e( \langle \textbf{b},\xi \rangle)\right| |\F(g_{\delta})(\xi)|\d\xi.
\end{equation}
Together with \eqref{wmbncGerror} and Lemma \ref{l2twistl1lem} we obtain
\begin{equation*}
\begin{aligned}
&\frac{1}{\# B}\left| \left| \sum_{\textbf{b}\in B} \theta_{\textbf{b}} f(\phi_{\textbf{b}}(x)) e( \langle \textbf{a},\textbf{b}\rangle) \right| \right|_{L^2(X)}\\
& \ll \frac{1}{\# B}+  \max_{|\textbf{t}|_{\infty}<\delta, x\in X} |f(x)-f(\phi_{\textbf{t}}(x))| \max_{\textbf{b}\in B}|\theta_{\textbf{b}}| \\
&+ \frac{(d(B))^d}{\#B} \gamma_f(d(B)) \int_{\R^d} \left|\sum_{\textbf{b}\in B} \theta_{\textbf{b}} e( \langle \textbf{b},\xi \rangle)\right| |\F(g_{\delta})(\xi)| \d\xi.
\end{aligned}
\end{equation*} 
for $\delta>0$.
\end{proof}

The rates \eqref{sparseequilem1rates} can be useful when one has a good understanding of
\begin{equation*}
\left|\sum_{\textbf{b}\in B} \theta_{\textbf{b}} e(-\frac{1}{d(B)} \langle \textbf{k},\textbf{b}\rangle) \right|.
\end{equation*}
 This is interesting even in the simplest (non-trivial) case when $B= [-N,N]^d \cap \Z^d$ and $\theta_{\textbf{b}}\equiv 1$, 
and $\phi$ is an $\R^d$-action. We will use this case to show weak mixing and effective equidistribution of the restricted map.

\begin{lemma}\label{Nexpsumlem}
Let $B= [-N,N]^d \cap \Z^d$ and $\Theta=(\theta_{\textbf{b}})\equiv 1$, then, for $\delta\in (0,1)$, we have
\begin{equation}\label{Nexpsumlemclaim}
\sum_{\textbf{k}=(k_1,...,k_d)\in \Z^d} S_{B}^{\Theta} \left(\frac{1}{N}\textbf{k}\right) \prod_{j=1}^d |\F_{d(B)}(\chi_{\delta})(k_j)|\ll \delta^{-2d}.
\end{equation}
\end{lemma}
\begin{proof}
It suffices to prove the statement for $d\!\!=\!\!1$, since, for 
$\textbf{k}=(k_1,...,k_d)
\in \Z^d$,
\begin{equation*}
S_{B}^{\Theta} \left(\frac{1}{N}\textbf{k}\right)=\prod_{j=1}^d \sum_{n=1}^{N} e\left( \frac{nk_j}{N} \right) .
\end{equation*}
So let $d=1$, then, for $k\in \Z$, we have
\begin{equation*}
S_N\left(\frac{k}{N}\right)=\left|\sum_{n=-N}^{N} e\left( \frac{nk}{N} \right)\right| =\begin{cases}
& 2N+1 \;\;\; \textit{if } k=jN \textit{ for some } j\in \Z,\\
& 1 \;\;\; \textit{otherwise.}
\end{cases},
\end{equation*}
Hence, using the Fourier coefficient from Lemma \ref{fouriercoeffchilem},
\begin{align*}
&\sum_{k\in \Z} S_N\left(\frac{k}{N}\right) |\F_{N}(\chi_{\delta})(k)|=\delta^{-1}+N\sum_{k\in \Z} |\F_{N}(\chi_{\delta})(Nk)|\\
&\ll \delta^{-1}+N\sum_{k\geq 1} \frac{2}{Nk^2\pi^2\delta^2} \left|1-\cos \left( 2\pi k\delta\right)\right|
\ll \delta^{-2}. 
\end{align*}
\end{proof}

\begin{remark}\label{SNrem}
We can insist on using Fourier transform instead of Fourier coefficients in this case, i.e. compute
\begin{equation*}
\int_{\R} S_{N} \left(\xi \right) |\F(\chi_{\delta})(\xi)| \d\xi
\end{equation*}
instead of \eqref{Nexpsumlemclaim}. However then one obtains a worse bound, since
\begin{equation*}
S_{N} \left(\xi \right)=\left|\sum_{n=1}^{N} e\left( n\xi \right)\right| \ll \min(N,d(\xi,\Z)^{-1}),
\end{equation*}
the same calculations as above only yield
\begin{equation*}
\int_{\R} S_{N} \left(\xi \right) |\F(\chi_{\delta})(\xi)| \d\xi \ll \delta^{-2} \log(N).
\end{equation*}
Of the examples we consider, this is the only one where it makes a difference whether we use Fourier transform or Fourier series.
\end{remark}

\begin{proposition}\label{time1twistprop}
If $\phi$ has decaying twisted integrals in $L^2$ of rate $\gamma$, then, for each $f\in L^2_0\cap \B$ and $\textbf{a}\in \R^d$, we have
\begin{equation}\label{time1twistdecrates}
\begin{aligned}
\frac{1}{(2N)^d} & \left| \left| \sum_{\textbf{k}\in [-N,N]^d \cap \Z^d} f(\phi_{\textbf{k}}(x)) e( \langle \textbf{a},\textbf{k}\rangle) \right| \right|_{L^2(X)}\\
& \ll \frac{1}{N^d}+  \max_{|\textbf{t}|_{\infty}<\delta, x\in X} |f(x)-f(\phi_{\textbf{t}}(x))| + \delta^{-2d} \gamma_f(N). 
\end{aligned}
\end{equation}
\end{proposition}
\begin{proof}
Let $B= [-N,N]^d \cap \Z^d$ and $\Theta=(\theta_{\textbf{b}})\equiv 1$, then $d(B)=N$ and by Lemma \ref{Nexpsumlem}, for $\delta\in (0,1)$, we have
\begin{equation*}
\sum_{\textbf{k}=(k_1,...,k_d)\in \Z^d} S_{B}^{\Theta} \left(\frac{1}{N}\textbf{k}\right) \prod_{j=1}^d |\F_{N}(\chi_{\delta})(k_j)|\ll \delta^{-2d}.
\end{equation*}
Using this in \eqref{sparseequilem1rates} we obtain
\begin{align*}
\frac{1}{(2N)^d}&\left| \left| \sum_{\textbf{b}\in B} f(\phi_{\textbf{b}}(x)) e( \langle \textbf{a},\textbf{b}\rangle) \right| \right|_{L^2(X)} \\
& \ll \frac{1}{N^d}+  \max_{|\textbf{t}|_{\infty}<\delta, x\in X} |f(x)-f(\phi_{\textbf{t}}(x))| + \delta^{-2d} \gamma_f(N).
\end{align*}

\end{proof}

\begin{lemma}\label{wmtime1wmlem}
Suppose condition (N1) is satisfied. If $\phi^{(\textbf{1})}$ weak mixing of rate $\alpha=\alpha^{\phi^{(\textbf{1})}}$, then for $f,g\in \B\cap L^2_0$ and for each $N,K\geq 1$ we have
\begin{equation}\label{wmtime1wmrates}
\begin{aligned}
&\left|\frac{1}{(2N)^d} \int_{[-N,N]^d} | \langle f, g(\phi_{\textbf{t}}) \rangle | \d\textbf{t}-\frac{1}{K^d} \sum_{\textbf{k}\in [0,K-1]^d} \alpha_{f,g(\phi_{\frac{1}{K}\textbf{k}})}(N) \right|\\
&\ll ||f||_{L^{\infty}} ||g||_{\B}  \max_{\textbf{t}\in \left[0,\frac{1}{K}\right]^d, x\in X} d(x,\phi_{\textbf{t}}(x))^{\rho}.
\end{aligned}
\end{equation}
\end{lemma}
\begin{proof}
For $K\geq 1$, to be chosen later, set
\begin{equation*}
g_{\textbf{k}}=g(\phi_{\frac{1}{K} \textbf{k}}) \;\;\; \textbf{k}\in [0,K-1]^d\cap \Z^d.
\end{equation*}
We have
\begin{align*}
&\left|\frac{1}{K^d} \sum_{\textbf{k}\in [0,K-1]^d} \sum_{\textbf{n}\in [-N,N]^d}  | \langle f, g_{\textbf{k}}(\phi_{\textbf{n}}) \rangle | - \int_{\textbf{t}\in [-T,T]^d} | \langle f, g(\phi_{\textbf{t}}) \rangle | \d\textbf{t} \right| \\
&\ll \frac{1}{K^d} \sum_{\textbf{k}\in [0,K-1]^d} \sum_{\textbf{n}\in [-N,N]^d} \left| | \langle f, g_{\textbf{k}}(\phi_{\textbf{n}}) \rangle | - K^d \int_{\textbf{t}\in [0,\frac{1}{K}]^d} | \langle f, g(\phi_{\textbf{t}+ \textbf{n} + \frac{1}{K}\textbf{k}}) \rangle | \d\textbf{t} \right|\\
&\ll ||f||_{L^{\infty}} ||g||_{\B}  \max_{\textbf{t}\in \left[0,\frac{1}{K}\right]^d, x\in X} d(x,\phi_{\textbf{t}}(x))^{\rho} N^d. 
\end{align*}
The conclusion \eqref{wmtime1wmrates} is immediate.
\end{proof}

We are now in a position to prove Theorem \ref{wmtime1thm}.

\begin{proof}[Proof of Theorem \ref{wmtime1thm}]
We already showed that $\phi$ is weak mixing if and only if $\phi^{(\textbf{1})}$ is. So let us focus on the rates of decay.

From now on, suppose $\phi$ is $\rho'$-Hölder continuous, and assumption (N) is satisfied.

(i) First assume that $\phi^{(\textbf{1)}}$ is polynomially weak mixing, say $\alpha^{\phi^{(\textbf{1})}}_{f,g}(N)\ll ||f||_{\B} ||g||_{\B} N^{-\kappa}$ for $f,g\in \B$. For $N,K\geq 1$, Lemma \ref{wmtime1wmlem} yields
\begin{align*}
\alpha^{\phi}_{f,g}(N)& \ll  ||f||_{\B} ||g||_{\B} K^{-\rho \rho'} + \frac{1}{K^d} \sum_{\textbf{k}\in [0,K-1]^d} \alpha_{f,g(\phi_{\frac{1}{K}\textbf{k}})}(N)\\
&\ll ||f||_{\B} ||g||_{\B}  ( K^{-\rho \rho'} + N^{-\kappa}),
\end{align*}
letting $K\rightarrow \infty$ yields $\alpha^{\phi}_{f,g}(N) \ll ||f||_{\B} ||g||_{\B} N^{-\kappa}$. Clearly, a fortiori, it also holds that
\begin{equation*}
\alpha^{\phi}_{f,g}(T) \ll ||f||_{\B} ||g||_{\B} T^{-\kappa} \;\;\;T>0.
\end{equation*}

(ii) On the other hand let $\phi$ be weak mixing, by Theorem \ref{wmtwistthm} twisted integrals decay. For $f\in L^2_0\cap \B$, Proposition \ref{time1twistprop} yields 
\begin{equation}\label{twisttime1ineq}
\gamma^{\phi^{(\textbf{1})}}_f(N) \ll \frac{1}{N^d}+ ||f||_{\B} \max_{|\textbf{t}|<\delta, x\in X} d(x,\phi_{\textbf{t}}(x))^{\rho} + \delta^{-2d} \gamma^{\phi}_f(N).
\end{equation}
Since $\phi$ is polynomially weak mixing, Theorem \ref{wmtwistthm} yields $\gamma^{\phi}_f(T)\ll ||f||_{\B} T^{-\kappa}$ for some $\kappa>0$. keeping in mind the Hölder continuity of $\phi$, \eqref{twisttime1ineq} becomes
\begin{equation*}
\gamma^{\phi^{(\textbf{1})}}_f(N) \ll \frac{1}{N^d}+ \delta^{\rho \rho'} ||f||_{\B} + \delta^{-2d} ||f||_{\B} N^{-\kappa}.
\end{equation*}
Letting $\delta=N^{-\frac{\kappa}{2d+\rho \rho'}}$ this becomes
\begin{equation*}
\gamma^{\phi^{(\textbf{1})}}_f(N) \ll ||f||_{\B} N^{-\frac{\rho \rho'\kappa}{2d+\rho \rho'}},
\end{equation*}
and applying Theorem \ref{wmtwistthm} for the $\Z^d$-action $\phi^{(\textbf{1})}$ yields that $\phi^{(\textbf{1})}$ is polynomially weak mixing. 
\end{proof}

\section{Pointwise Twisted Integrals}

This subsection will be devoted to proving the main proposition showing 
the effective rate of
decay for twisted integrals in Proposition \ref{vprop1}. 
The argument is analogous to the one that first appeared in \cite{venkateshtwist}, which shows polynomial decay of twisted integrals for horocycle flow. 

\begin{lemma}\label{averagetwistlem}
For $f\in L^{\infty}$ and $\textbf{a}\in \R^d$ (or $[0,1)^d$) we have
\begin{align*}
\left|
\begin{aligned}\frac{1}{(2T)^d} &\int_{[-T,T]^d} f(\phi_{\textbf{t}} x) e( \langle \textbf{a} , \textbf{t}\rangle) \d m(\textbf{t})\\
& - \frac{1}{(2T)^d(2H)^d} \int_{[-T,T]^d} \int_{[-H,H]^d} f(\phi_{\textbf{t}+\textbf{h}} x) e( \langle \textbf{a} , \textbf{t}+\textbf{h}\rangle) \d m(\textbf{h}) \d m(\textbf{t})
\end{aligned}
\right| \ll \frac{H}{T} ||f||_{L^{\infty}}
\end{align*}
for every $x\in X$ and $T,H>0$.
\end{lemma}
\begin{proof}
For every $x\in X$ and $T,H>0$
\begin{align*}
&\left|
\begin{aligned}\frac{1}{(2T)^d} &\int_{[-T,T]^d} f(\phi_{\textbf{t}} x) e( \langle \textbf{a} , \textbf{t}\rangle) \d m(\textbf{t})\\
& - \frac{1}{(2T)^d(2H)^d} \int_{[-T,T]^d} \int_{[-H,H]^d} f(\phi_{\textbf{t}+\textbf{h}} x) e( \langle \textbf{a} , \textbf{t}+\textbf{h}\rangle) \d m(\textbf{h}) \d m(\textbf{t})
\end{aligned}
\right|\\
&\ll \frac{1}{T^dH^d} \int_{[-H,H]^d} \left|\int_{[-T,T]^d} f(\phi_{\textbf{t}} x) e( \langle \textbf{a} , \textbf{t}\rangle) \d m(\textbf{t})  - \int_{[-T,T]^d} f(\phi_{\textbf{t}+\textbf{h}} x) e( \langle \textbf{a} , \textbf{t}+\textbf{h}\rangle) \d m(\textbf{t}) \right| \d m(\textbf{h}) \\
&\ll \frac{1}{T^dH^d} \int_{[-H,H]^d} \left|\int_{[-T,T]^d} f(\phi_{\textbf{t}} x) e( \langle \textbf{a} , \textbf{t}\rangle) \d m(\textbf{t})  - \int_{[-T,T]^d+\textbf{h}} f(\phi_{\textbf{t}} x) e( \langle \textbf{a} , \textbf{t}\rangle) \d m(\textbf{t}) \right| \d m(\textbf{h}) \\
&\ll \frac{1}{T^dH^d} ||f||_{L^{\infty}} \int_{[-H,H]^d} m([-T,T]^d\Delta ([-T,T]^d+\textbf{h})) \d m(\textbf{h}) \\
& \ll \frac{1}{T^dH^d} ||f||_{L^{\infty}} \int_{[-H,H]^d} 2 T^{d-1} |\textbf{h}|_1 \d m(\textbf{h}) \\
&\ll \frac{H}{T} ||f||_{L^{\infty}}.
\end{align*}

\end{proof}

\begin{proof}[Proof of Proposition \ref{vprop1}]
By Lemma \ref{averagetwistlem}, it is enough to show
\begin{align*}
&\left| \frac{1}{(2T)^d(2H)^d} \int_{[-T,T]^d} \int_{[-H,H]^d} f(\phi_{\textbf{t}+\textbf{h}} x) e( \langle \textbf{a} ,\textbf{t} + \textbf{h} \rangle) \d m(\textbf{h}) \d m(\textbf{t}) \right| \\
& \ll \alpha_{f,f}(H)^{\frac{1}{2}} + \frac{1}{(2H)^d} \left( \int_{[-H,H]^d} \int_{[-H,H]^d} \beta_{f(\phi_{\textbf{h}_1}) \overline{f(\phi_{\textbf{h}_2})},x}(\textbf{t})\d m(\textbf{h}_1)  \d m(\textbf{h}_2) \right)^{\frac{1}{2}}. 
\end{align*}

Let $x\in X$, using Cauchy-Schwartz inequality with functions $1$ and $ \int_{[-H,H]^d} f(\phi_{\textbf{t}+\textbf{h}} x) e( \langle \textbf{a} ,\textbf{t} +\textbf{h} \rangle) \d m(\textbf{h}) $ we obtain
\begin{align*}
&\frac{1}{(2T)^d (2H)^d} \left|  \int_{[-T,T]^d} \int_{[-H,H]^d} f(\phi_{\textbf{t}+\textbf{h}} x) e( \langle \textbf{a} ,\textbf{t} +\textbf{h} \rangle) \d m(\textbf{h}) \d m(\textbf{t}) \right| \\
& \ll \frac{1}{T^{\frac{d}{2}} H^d} \left| \left( \int_{[-T,T]^d} \int_{[-H,H]^d} \int_{[-H,H]^d} e( \langle \textbf{a} ,\textbf{h}_1-\textbf{h}_2 \rangle) f(\phi_{\textbf{t}+\textbf{h}_1} x) \overline{f(\phi_{\textbf{t}+\textbf{h}_2} x)} \d m(\textbf{h}_1)  \d m(\textbf{h}_2) \d m(\textbf{t}) \right)^{\frac{1}{2}} \right| \\
& \ll \frac{1}{T^{\frac{d}{2}} H^d} \left( \int_{[-H,H]^d} \int_{[-H,H]^d} e( \langle \textbf{a} ,\textbf{h}_1-\textbf{h}_2 \rangle) \left| \int_{[-T,T]^d} f(\phi_{\textbf{t}+\textbf{h}_1} x) \overline{f(\phi_{\textbf{t}+\textbf{h}_2} x)} \d m(\textbf{t})  \right| \d m(\textbf{h}_1)  \d m(\textbf{h}_2)  \right)^{\frac{1}{2}} \\
& \ll \frac{1}{T^{\frac{d}{2}} H^d} \left( \int_{[-H,H]^d} \int_{[-H,H]^d} \left| \int_{[-T,T]^d} f(\phi_{\textbf{t}+\textbf{h}_1} x) \overline{f(\phi_{\textbf{t}+\textbf{h}_2} x)} \d m(\textbf{t})  \right| \d m(\textbf{h}_1)  \d m(\textbf{h}_2)  \right)^{\frac{1}{2}} \\
& \ll \frac{1}{T^{\frac{d}{2}} H^d} \left( 
\int_{[-H,H]^d} \int_{[-H,H]^d} \begin{aligned}
&\left| T^d \int_X f(\phi_{\textbf{h}_1} x) \overline{f(\phi_{\textbf{h}_2} x)} \d\mu(x) \right| \\
& + T^d \beta_{f(\phi_{\textbf{h}_1}) \overline{f(\phi_{\textbf{h}_2})},x}(T)
\end{aligned} \d m(\textbf{h}_1)  \d m(\textbf{h}_2) 
\right)^{\frac{1}{2}} \\
&\ll \frac{1}{H^d} \left( \int_{[-H,H]^d} \int_{[-H,H]^d} \left| \langle f(\phi_{\textbf{h}_1}),f(\phi_{\textbf{h}_2}) \rangle \right| \d m(\textbf{h}_1)  \d m(\textbf{h}_2) \right)^{\frac{1}{2}}\\
&+  \frac{1}{H^d} \left( \int_{[-H,H]^d} \int_{[-H,H]^d} \beta_{f(\phi_{\textbf{h}_1}) \overline{f(\phi_{\textbf{h}_2})},x}(\textbf{t})\d m(\textbf{h}_1)  \d m(\textbf{h}_2) \right)^{\frac{1}{2}}\\
&\ll \frac{1}{H^{\frac{d}{2}}} \left( \int_{[-2H,2 H]^d} \left| \langle f(\phi_{\textbf{u}}),f \rangle \right| \d m(\textbf{u})\right)^{\frac{1}{2}}\\
&+ \frac{1}{H^d} \left( \int_{[-H,H]^d} \int_{[-H,H]^d} \beta_{f(\phi_{\textbf{h}_1}) \overline{f(\phi_{\textbf{h}_2})},x}(T)\d m(\textbf{h}_1)  \d m(\textbf{h}_2) \right)^{\frac{1}{2}},
\end{align*}

where in the last step we substituted $\textbf{u}=\textbf{h}_1-\textbf{h}_2$. 
For the prior term, use \eqref{wmratesdef} to obtain 
\begin{equation*}
\frac{1}{H^{\frac{d}{2}}} \left( \int_{[-H,H]^d} \left| \langle f(\phi_{\textbf{u}}),f \rangle \right| \right)^{\frac{1}{2}} \ll \alpha_{f,f}(H)^{\frac{1}{2}}. 
\end{equation*}
\end{proof}

To obtain polynomial bounds, we first have to analyse $\beta_{f,x}(\phi_{\textbf{h}})(T)$ for fixed $T>0$. This will help us control the term
\begin{equation*}
\int_{[-H,H]^d} \int_{[-H,H]^d} \beta_{f(\phi_{\textbf{h}_1}) \overline{f(\phi_{\textbf{h}_2})},x}(T)\d m(\textbf{h}_1)  \d m(\textbf{h}_2).
\end{equation*}

\begin{lemma}\label{betalem1}
Let $f\in L^2_0\cap L^{\infty}$, $x\in X$ and $\textbf{h}\in \R^d$ (or $\Z^d$) then
\begin{equation*}
\beta_{f(\phi_{\textbf{h}}),x}(\textbf{t}) \ll |\textbf{h}|_{\infty} T^{-1} 
||f||_{L^{\infty} }
+ \beta_{f,x}(T+|\textbf{h}|_{\infty}) \;\;\; \textit{for } T\geq |h|_{\infty}.
\end{equation*}
\end{lemma}
\begin{proof}
The claim is immediate from
\begin{equation*}
\left|\int_{[-T,T]^d}  f(\phi_{\textbf{h}+\textbf{t}}(x)) \d m(\textbf{t}) - \int_{[-T-|\textbf{h}|_{\infty},T+|\textbf{h}|_{\infty}]^d} f(\phi_{\textbf{t}}(x)) \d m(\textbf{t})\right| \ll |\textbf{h}|_{\infty} T^{d-1} ||f||_{L^{\infty}}.
\end{equation*}

\end{proof}

\begin{proof}[Proof of Corollary \ref{quantpolytwistcor}]
First, we claim 
\begin{equation}\label{quantpolytwistpropclaim}
\begin{aligned}
\int_{[-H,H]^d} \int_{[-H,H]^d} & \beta_{f(\phi_{\textbf{h}_1}) \overline{f(\phi_{\textbf{h}_2})},x}(T)\d m(\textbf{h}_1)  \d m(\textbf{h}_2)\\
& \ll T^{-1} ||f||_{L^{\infty}}^2 H^{3d} + ||f||_{\B}^2 T^{-\delta_2} H^{d(\mathcal{K}+1)}
\end{aligned} 
\end{equation}
for $T\geq H$. Indeed, for each $\textbf{h}_1,\textbf{h}_2\in [-H,H]^d$, by Lemma \ref{betalem1}, we have
\begin{equation*}
\beta_{f(\phi_{\textbf{h}_1}) \overline{f(\phi_{\textbf{h}_2})},x}(T) \ll |\textbf{h}_2|_{\infty} T^{-1} ||f||_{L^{\infty}}^2 + \beta_{f(\phi_{\textbf{h}_1-\textbf{h}_2}) \overline{f},x}(T+|\textbf{h}_2|_{\infty}).
\end{equation*}
Hence 
\begin{align*}
\int_{[-H,H]^d} \int_{[-H,H]^d} & \beta_{f(\phi_{\textbf{h}_1}) \overline{f(\phi_{\textbf{h}_2})},x}(T)\d m(\textbf{h}_1)  \d m(\textbf{h}_2)\\
&\ll T^{-1} ||f||_{L^{\infty}}^2 H^d \int_{[-H,H]^d} |\textbf{h}|_{\infty} \d m(\textbf{h})\\
& + \int_{[-H,H]^d} \int_{[-H,H]^d} \beta_{f(\phi_{\textbf{h}_1-\textbf{h}_2}) \overline{f},x}(T+|\textbf{h}_2|_{\infty})\d m(\textbf{h}_1)  \d m(\textbf{h}_2)\\
&\ll T^{-1} ||f||_{L^{\infty}}^2 H^{3d} + ||f||_{\B}^2 T^{-\delta_2} \int_{[-H,H]^d} \int_{[-H,H]^d} (|\textbf{h}_1-\textbf{h}_2|_1)^{\mathcal{K}} \d m(\textbf{h}_1) \d m(\textbf{h}_2)\\
&\ll T^{-1} ||f||_{L^{\infty}}^2 H^{3d} + ||f||_{\B}^2 T^{-\delta_2} H^{d(\mathcal{K}+1)}.
\end{align*}

By Proposition \ref{vprop1} and \eqref{quantpolytwistpropclaim} we have
\begin{align*}
&\left|\frac{1}{(2T)^d} \int_{[-T,T]^d} f(\phi_{\textbf{t}} x) e( \langle\textbf{a} , \textbf{t}\rangle) \d m(\textbf{t}) \right| \\
&\ll \frac{H}{T} ||f||_{L^{\infty}} + ||f||_{\B} H^{-\frac{\delta_1}{2}} + \frac{1}{(2H)^d} \left( T^{-1} ||f||_{L^{\infty}}^2 H^{3d} + ||f||_{\B}^2 T^{-\delta_2} H^{d(\mathcal{K}+1)} \right)^{\frac{1}{2}}\\
&\ll ||f||_{\B}\left(T^{-1} H + H^{-\frac{\delta_1}{2}} + T^{-\frac{1}{2}} H^{\frac{d}{2}} + T^{-\frac{\delta_2}{2}} H^{\frac{d\mathcal{K}}{2}}\right).
\end{align*}
The Ansatz $H=T^{\kappa}$ yields
\begin{equation*}
\left|\frac{1}{(2T)^d} \int_{[-T,T]^d} \psi(\textbf{t}) f(\phi_t x) \d m(\textbf{t}) \right| 
\ll ||f||_{\B} \left(T^{-1+\kappa} + T^{-\kappa\frac{\delta_1}{2}} + T^{-\frac{1}{2} + \kappa \frac{d}{2}} + T^{-\frac{\delta_2}{2} + \kappa \frac{\mathcal{K}d}{2}} \right).
\end{equation*}

The claim follows with
\begin{equation*}
\delta=\max_{\kappa>0} \min \left(1-\kappa,\kappa\frac{\delta_1}{2},\frac{1}{2} - \kappa \frac{d}{2},\frac{\delta_2}{2} - \kappa \frac{\mathcal{K}d}{2}\right).
\end{equation*}
In order to compute $\delta$, consider the corresponding lines given by
\begin{equation*}
I(\kappa)=1-\kappa, \; II(\kappa)=\kappa\frac{\delta_1}{2}, \; III(\kappa)=\frac{1}{2}- \kappa \frac{d}{2}, \; \textit{ and } \; IV(\kappa)=\frac{\delta_2}{2} - \kappa \frac{\mathcal{K}d}{2},
\end{equation*}
in $\R^2$. Since only $II$ is increasing and all the other lines are decreasing, but have positive displacement, $\max_{\kappa>0} \min(I(\kappa), II(\kappa), III(\kappa), IV(\kappa))$ can be found at the first intersection of $II$ with another line. The points of intersection of $I$, $III$ and $IV$ with $II$ are
\begin{equation*}
\kappa_1=\frac{1}{1+\frac{\delta_1}{2}}=\frac{2}{2+\delta_1}, \;\kappa_3= \frac{\frac{1}{2}}{\frac{d}{2}+\frac{\delta_1}{2}}=\frac{1}{d+\delta_1}, \; \textit{ and } \kappa_4= \frac{\frac{\delta_2}{2}}{\frac{\mathcal{K}d}{2}+ \frac{\delta_1}{2}}=\frac{\delta_2}{\mathcal{K}d+\delta_1}.
\end{equation*}
Thus
\begin{equation*}
\delta=\min\left(\frac{\delta_1}{2+\delta_1}, \frac{\delta_1}{2(d+\delta_1)}, \frac{\delta_1\delta_2}{2(\mathcal{K}d+\delta_1)}\right).
\end{equation*}
\end{proof}

\section{Effective equidistribution}

Here we obtain the pointwise analogues of the results of Section \ref{restrictsec}. This 
gives much stronger estimates, even if the proofs remain almost the same.

\begin{lemma}\label{pointtwistl1lem}
For $x\in X$, $f\in L^2_0$, $T>0$ and $\psi\in C^{\infty}_c$ with $\supp \psi \subset [-A,A]^d$, 
\begin{equation*}
\frac{1}{(2T)^d} \left| \int_{[-T,T]^d} f(\phi_{\textbf{t}}(x)) \psi(\textbf{t}) \d\textbf{t} \right| \ll ||\F_A(\psi)||_{l^1} \gamma_{f,x}(T).
\end{equation*}
\end{lemma}
\begin{proof}
By triangle inequality we have
\begin{align*}
\left| \int_{[-T,T]^d} f(\phi_{\textbf{t}}(x)) \psi(\textbf{t}) \d\textbf{t} \right|& \leq \left| \int_{[-T,T]^d} \sum_{\textbf{k}\in \Z^d} \F_A(\psi)(\textbf{k}) e(\frac{1}{A} \langle \textbf{k}, \textbf{t} \rangle) f(\phi_{\textbf{t}}(x)) \d\textbf{t} \right| \\
& \leq \sum_{\textbf{k}\in \Z^d} |\F_A(\psi)(\textbf{k})| \left|\int_{[-T,T]^d} e(\frac{1}{A} \langle \textbf{k}, \textbf{t} \rangle) f(\phi_{\textbf{t}}(x)) \d\textbf{t} \right|.
\end{align*}

\end{proof}

\begin{proposition}\label{sparseequilem1point}
Let $B\subset \R^d$ be a finite set and $\Theta=(\theta_{\textbf{b}})$ be real (possibly negative) numbers. If $\phi$ has decaying twisted integrals at some $x\in X$ of rate $\gamma_{f,x}$, then, for each $f\in L^2_0$ and $\textbf{a}\in \R^d$, 
\begin{align*}
\frac{1}{\# B}\left|\sum_{\textbf{b}\in B} \theta_{\textbf{b}} f(\phi_{\textbf{b}}(x)) e( \langle \textbf{a},\textbf{b}\rangle) \right| & \ll \frac{1}{\# B}+  \max_{|\textbf{t}|_{\infty}<\delta, x\in X} |f(x)-f(\phi_{\textbf{t}}(x))| \max_{\textbf{b}\in B}|\theta_{\textbf{b}}|\\
& + \frac{(d(B))^d}{\#B} \gamma_{f,x}(d(B)) \int_{\R^d} S_{B}^{\Theta} \left(\xi \right) \prod_{j=1}^d |\F(\chi_{\delta})(\xi_j)| \d\xi.
\end{align*}

and
\begin{align*}
\frac{1}{\# B} \left| \sum_{\textbf{b}\in B} \theta_{\textbf{b}} f(\phi_{\textbf{b}}(x)) e( \langle \textbf{a},\textbf{b}\rangle) \right|& \ll \frac{1}{\# B}+  \max_{|\textbf{t}|_{\infty}<\delta, x\in X} |f(x)-f(\phi_{\textbf{t}}(x))| \max_{\textbf{b}\in B}|\theta_{\textbf{b}}|\\
& + \frac{(d(B))^d}{\#B} \gamma_{f,x}(d(B)) \sum_{\textbf{k}=(k_1,...,k_d)\in \Z^d} S_{B}^{\Theta} \left(\frac{1}{d(B)}\textbf{k}\right) \prod_{j=1}^d |\F_{d(B)}(\chi_{\delta})(k_j)|.
\end{align*}
for all $\delta>0$.
\end{proposition}
\begin{proof}
Analogous to Proposition \ref{sparseequilem}.
\end{proof}

For $B=[-N,N]^d\cap \Z^d$ and $\Theta\equiv 1$, analogously to above, we obtain

\begin{proposition}\label{pointtime1twistprop}
If $\phi$ has decaying twisted integrals at $x$ of rate $\gamma_{f,x}$, then, for each $f\in L^2_0\cap \B$ and $\textbf{a}\in \R^d$, 
\begin{equation}\label{pointtime1twistdecrates}
\frac{1}{N^d} \left| \sum_{\textbf{k}\in [-N,N]^d \cap \Z^d} f(\phi_{\textbf{k}}(x)) e( \langle \textbf{a},\textbf{k}\rangle) \right| \ll \frac{1}{N^d}+  \max_{|\textbf{t}|_{\infty}<\delta, x\in X} |f(x)-f(\phi_{\textbf{t}}(x))| + \delta^{-2d} \gamma_{f,x}(N).
\end{equation}
\end{proposition}
\begin{proof}
Analogous to Proposition \ref{time1twistprop}.
\end{proof}

\begin{proof}[Proof of Theorem \ref{unithm}]
Let $f\in L^1_0$ be uniformly continuous and bounded. By Proposition \ref{vprop1}, for every $H>0$, we have
\begin{equation*}
\gamma_{f,x}(T) \ll  \frac{H}{T} ||f||_{L^{\infty}} + \alpha_{f,f}(H)^{\frac{1}{2}} 
+ \frac{1}{H^d} \left( \int_{[-H,H]^d} \int_{[-H,H]^d} \beta_{f(\phi_{\textbf{h}_1}) \overline{f(\phi_{\textbf{h}_2})},x}(T)\d m(\textbf{h}_1)  \d m(\textbf{h}_2) \right)^{\frac{1}{2}}.
\end{equation*}
We first claim that $\gamma_{f,x}(T) \rightarrow 0$ as $T\rightarrow \infty$. 

Indeed, for $\epsilon>0$, first let $H>0$ be so big that $\alpha_{f,f}(H)<\epsilon^2$. Let $T_0'$ be so big that\footnote{ This is possible since the map $(\textbf{h}_1,\textbf{h}_2,T)\mapsto \beta_{f(\phi_{\textbf{h}_1}) \overline{f(\phi_{\textbf{h}_2})},x}(T)$ is continuous.}
$\beta_{f(\phi_{\textbf{h}_1}) \overline{f(\phi_{\textbf{h}_2})},x}(T) < \epsilon^2 \;\;\;\forall \textbf{h}_1,\textbf{h}_2\in [-H,H]^d, T>T_0'$, then we have $\gamma_{f,x}(T)\ll \epsilon$ for $T>T_0=\max(T'_0, \epsilon^{-1} H ||f||_{L^{\infty}})$.
Hence, $\gamma_{f,x}(T) \rightarrow 0$ as $T\rightarrow \infty$. \\
Setting $\textbf{a}=\textbf{0}$ in Proposition \ref{pointtime1twistprop} with yields\footnote{Because $\max_{|\textbf{t}|_{\infty}<\delta, x\in X} |f(x)-f(\phi_{\textbf{t}}(x))|\rightarrow 0$ as $\delta>0$.}
\begin{equation}\label{unithmeq1}
\frac{1}{N^d} \left| \sum_{\textbf{k}\in [-N,N]^d \cap \Z^d} f(\phi_{\textbf{k}}(x)) \right|\rightarrow 0\;\;\;\as{N},
\end{equation}
which shows the first conclusion.

If $\phi$ is weak mixing and uniquely ergodic, then $B_{\phi}(\mu)=X$. By what we have just shown $B_{\phi^{(\textbf{1})}}(\mu)=X$ and $\phi^{(\textbf{1})}$ is also uniquely ergodic.

The 'furthermore' part follows directly from Corollary \ref{quantpolytwistcor} and 
Proposition \ref{pointtime1twistprop} with $\textbf{a} = \textbf{0}$.
\end{proof}

\begin{proof}[Proof of Theorem \ref{classthm}]
The set of invariant probability measures is convex, and its extreme points are the invariant and ergodic probability measures, furthermore from Theorem \ref{unithm} it follows that whenever $\mu$ is an invariant and ergodic probability measure for $\phi$, then it is also invariant and ergodic for $\phi^{(\textbf{1})}$.\\
Therefore, it suffices to show the following: If $\mu$ is an invariant and ergodic probability measure for $\phi^{(\textbf{1})}$, then it is also invariant and ergodic for $\phi$. Note that in particular the basin $B_{\phi^{(\textbf{1})}}(\mu)\not= \emptyset$ is non-empty, let $x\in B_{\phi^{(\textbf{1})}}(\mu)$ be a $\phi^{(\textbf{1})}$-generic point for $\mu$.

If there is an invariant and ergodic probability measure $\nu$ for $\phi$ such that $x\in B_{\phi}(\nu)$, then by Theorem \ref{unithm} it holds that $x\in B_{\phi^{(\textbf{1})}}(\nu)$, i.e. $x$ is $\phi^{(\textbf{1})}$-generic point for $\nu$. It follows that $\mu=\nu$.

If $x$ is not $\phi$-generic for any measure, then there are\footnote{The family of probability measures $\nu_t=\frac{1}{(2T)^d} \int_{[-T,T]^d} \delta_{\phi_{\textbf{t}}(x)} \d \textbf{t}$, does not converge. But by compactness $\{\nu_t\}_{t>0}$ has accumulation points, say there are invariant and ergodic probability measures $\nu^{(1)}$ and $\nu^{(2)}$ for $\phi$, as well as sequences $\s{T^{(1)}}{n}$ and $\s{T^{(2)}}{n}$ such that
\begin{equation*}
\nu_{T^{(1)}_n} \Rightarrow \nu^{(1)} \;\;\;\as{n}, \; \textit{ while }\; \nu_{T^{(2)}_n} \Rightarrow \nu^{(2)} \;\;\;\as{n},
\end{equation*}
where $\Rightarrow$ denotes weak convergence. We may even assume that $T^{(1)}_n, T^{(2)}_n\in\N$.}
invariant and ergodic probability measures $\nu^{(1)}$ and $\nu^{(2)}$ for $\phi$, as well as sequences $\s{T^{(1)}}{n}$ and $\s{T^{(2)}}{n}$ of natural numbers such that 
\begin{equation*}
\frac{1}{(2T^{(1)}_n)^d} \int_{[-T^{(1)}_n,T^{(1)}_n]^d} f(\phi_{\textbf{t}}(x)) \d \textbf{t} \rightarrow \int_X f \d\nu^{(1)} \;\;\;\as{n},
\end{equation*}
and
\begin{equation*}
\frac{1}{(2T^{(2)}_n)^d} \int_{[-T^{(2)}_n,T^{(2)}_n]^d} f(\phi_{\textbf{t}}(x)) \d \textbf{t} \rightarrow \int_X f \d\nu^{(2)} \;\;\;\as{n},
\end{equation*}
for all uniformly continuous bounded functions $f:X\rightarrow\R$. We claim that 
\begin{equation}\label{classeq1}
\frac{1}{(2T^{(1)}_n)^d} \sum_{\textbf{k}\in[-T^{(1)}_n,T^{(1)}_n]^d\cap\Z^d} f(\phi_{\textbf{k}}(x)) \rightarrow \int_X f \d\nu^{(1)} \;\;\;\as{n},
\end{equation}
and
\begin{equation*}
\frac{1}{(2T^{(2)}_n)^d} \sum_{\textbf{k}\in[-T^{(2)}_n,T^{(2)}_n]^d\cap \Z^d} f(\phi_{\textbf{k}}(x)) \rightarrow \int_X f \d\nu^{(2)} \;\;\;\as{n},
\end{equation*}
this will then contradict $x\in B_{\phi^{(\textbf{1})}}(\mu)$. We only show \eqref{classeq1}.

Denote $f^{(1)}=f-\int_X f \d\nu^{(1)}$, applying Proposition \ref{vprop1} yields
\begin{align*}
\gamma_{f^{(1)},x}(T^{(1)}_n) & \ll  \frac{H}{T^{(1)}_n} ||f^{(1)}||_{L^{\infty}} + \alpha_{f^{(1)},f^{(1)}}(H)^{\frac{1}{2}} \\
&+ \frac{1}{H^d} \left( \int_{[-H,H]^d} \int_{[-H,H]^d} \beta_{f^{(1)}(\phi_{\textbf{h}_1}) \overline{f^{(1)}(\phi_{\textbf{h}_2})},x}(T^{(1)}_n)\d m(\textbf{h}_1)  \d m(\textbf{h}_2) \right)^{\frac{1}{2}},
\end{align*}
for every $H>0$ and $n\geq 1$. As in the proof of Theorem \ref{unithm} it follows that $\gamma_{f^{(1)},x}(T^{(1)}_n)\rightarrow 0$ as $n \rightarrow 0$, and Proposition \ref{pointtime1twistprop} with $\textbf{a}=\textbf{0}$ yields 
\begin{equation*}
\frac{1}{(2T^{(1)}_n)^d} \left| \sum_{\textbf{k}\in [-T^{(1)}_n,T^{(1)}_n]^d \cap \Z^d} f^{(1)}(\phi_{\textbf{k}}(x)) \right|\rightarrow 0\;\;\;\as{n},
\end{equation*}
which is equivalent to \eqref{classeq1}.
\end{proof}

\begin{remark}
The proof of Theorem \ref{classthm} shows in fact the following stronger result: for every $x\in X$ and every continuous and bounded function $f:X\rightarrow\R$ it holds that
\begin{equation*}
\frac{1}{(2N)^d} \left| \int_{[-N,N]^d} f(\phi_{\textbf{t}}(x)) \d \textbf{t} - \int_{\textbf{k}\in [-N,N]^d\cap \Z^d} f(\phi_{\textbf{k}}(x)) \right| \rightarrow 0 \;\;\;\as{N}.
\end{equation*}
Indeed, for every subsequence $\s{N}{j}$ there is a further subsequence $(N_{j_l})_{l\geq 1}$ such that 
\begin{equation*}
\frac{1}{(2N_{j_l})^d} \int_{[-N_{j_l},N_{j_l}]^d} \delta_{\phi_{\textbf{t}}(x)} \d \textbf{t} \Rightarrow \mu \;\;\;\as{j},
\end{equation*}
for some invariant and ergodic measure $\mu$, where $\Rightarrow$ denotes weak convergence. The proof of Theorem \ref{classthm} then yields
\begin{equation*}
\frac{1}{(2N_{j_l})^d} \sum_{\textbf{k}\in [-N_{j_l},N_{j_l}]^d \cap \Z^d} \delta_{\phi_{\textbf{k}}(x)} \d \Rightarrow \mu \;\;\;\as{j},
\end{equation*}
and it follows that
\begin{equation*}
\frac{1}{(2N_{j_l})^d} \left| \int_{[-N_{j_l},N_{j_l}]^d} f(\phi_{\textbf{t}}(x)) \d \textbf{t} - \int_{\textbf{k}\in [-N_{j_l},N_{j_l}]^d\cap \Z^d} f(\phi_{\textbf{k}}(x)) \right| \rightarrow 0 \;\;\;\as{j}.
\end{equation*}
\end{remark}

\section{The sequence $n^{1+\epsilon}$} \label{1+epssec}

In this section, we shall investigate the sequence $\s{a}{n}=(n^{1+\epsilon})_{n\geq 0}$ and higher dimensional analogues.\\
For small enough $\epsilon_1,...,\epsilon_d>0$ we consider 
\begin{equation*}
B_N=\{ (n_1^{1+\epsilon_1},...,n_d^{1+\epsilon_d}) | n_j\in [0,N-1]^d\cap \Z^d \forall j=1,..,d\}
\end{equation*}
First we consider the case $d=1$, i.e. $\{n^{1+\epsilon} \;|\; n=0,...,N-1\}$ for some $\epsilon>0$. 

As in \cite{venkateshtwist}, the trick is to approximate $(n^{1+\epsilon_j})_{n=0}^{N-1}$ by arithmetic progressions on large intervals. Let $a>1$ to be chosen later and, for $j = 1,...,J(N)$ with $J(N)=\lceil N^{\frac{1}{a}} \rceil$, denote $k_j=\lceil j^{a} \rceil$.\\
For each $j = 1,...,J(N)$ let $B'_{N,j}=\{b'_n \;|\; k_{j-1} \leq n \leq \min(N,k_j)-1\}$, where
\begin{equation*}
b'_n=k_j^{1+\epsilon}+ (1+\epsilon) k_j^{\epsilon} (n-k_j) , \;\;\; \textit{ for } k_j\leq n <k_{j+1}, j\geq 1.
\end{equation*}
Note that for fixed $N_0\geq 1$ and all $n\geq 1$ we have
\begin{equation*}
|(N_0+n)^{1+\epsilon} - (N_0^{1+\epsilon}+(1+\epsilon)N_0^{\epsilon}n)| \leq \epsilon(1+\epsilon) N_0^{\epsilon-1}n^2.
\end{equation*}
Therefore, assuming $\phi$ to be Lipschitz, we have
\begin{align*}
&\left| \sum_{n=1}^N f(\phi_{n^{1+\epsilon}}(x))- \sum_{n=1}^N f(\phi_{b'_n}(x)) \right| 
\ll \sum_{j=1}^{J(N)} \sum_{n=k_{j-1}}^{\min(N,k_j)-1} |f(\phi_{n^{1+\epsilon}}(x)) - f(\phi_{b'_n}(x))| \\
&\ll ||f||_{\B} \sum_{j=1}^{J(N)} \sum_{l=0}^{\min(N,k_j)-k_{j-1}-1} k_{j-1}^{\epsilon-1} l^2
\ll ||f||_{\B} \sum_{j\leq N^{\frac{1}{a}}} j^{a(\epsilon-1)} \sum_{l\leq j^{a-1}} l^2\\
&
\ll ||f||_{\B} \sum_{j\leq N^{\frac{1}{a}}} j^{a(\epsilon+2)-3}
\ll ||f||_{\B} N^{\epsilon+2-\frac{2}{a}},
\end{align*}

this is $o(N)$ as long as $a<\frac{2}{1+\epsilon}$.

\begin{lemma}\label{b'lem}
For every $j=1,...,J(N)$ we have
\begin{equation*}
\int_{\R} S_{B'_{N,j}}(\xi) |\F(\chi_{\delta})(\xi)| \d\xi \ll  \frac{1}{\delta} \log(j). 
\end{equation*}
\end{lemma}

\begin{proof}
For $\xi\in \R$ we have
\begin{equation*}
S_{B'_{N,j}}(\xi)= \left|\sum_{b\in B'_{N,j}} e(\xi b)\right| \ll \left| \sum_{n=k_{j-1}}^{k_j-1} e( \xi b'_n) \right|
\ll \left| \sum_{n=0}^{\lfloor j^a - (j-1)^a \rfloor} e( \xi (1+\epsilon) k_{j-1}^{\epsilon} n)  \right|.
\end{equation*}

If $\xi (1+\epsilon)k_j^{\epsilon}\not\in \Z$ we have
\begin{equation}\label{ejest}
E_j(\xi):=\left| \sum_{n=0}^{\lfloor \min(N,j^a) - (j-1)^a \rfloor} e( \xi (1+\epsilon)k_j^{\epsilon}n) \right| \ll \frac{1}{\{ \xi (1+\epsilon)k_j^{\epsilon} \}},
\end{equation}
where $\{x\} =x-\lfloor x\rfloor$ denotes the fractional part. For $\omega\in (0,1)$ to be chosen later let
\begin{equation*}
A_j^{\omega}:= \bigcup_{l\geq 0} \underbrace{\left[ \frac{l-\omega}{1+\epsilon} k_j^{-\epsilon}, \frac{l+\omega}{1+\epsilon} k_j^{-\epsilon}\right]}_{=:A_{j,l}^{\omega}},
\end{equation*}
for $\xi\in A_j^{\omega}$ we will estimate trivially $E_j(\xi)\ll j^{a-1}$, while for $\xi\not\in A_j^{\omega}$ we use \eqref{ejest}. We split the integral as
\begin{align*}
\int_{\R} & E_j\left( \xi \right) |\F(\chi_{\delta})(\xi)| \d\xi \\
&\ll \underbrace{ \int_{A_j^{\omega}\cap [-\delta^{-1},\delta^{-1}]} E_j(\xi) |\F(\chi_{\delta})(\xi)| \d\xi}_{=I}+ \underbrace{\int_{(A_j^{\omega})^c\cap [-\delta^{-1},\delta^{-1}]} E_j(\xi) |\F(\chi_{\delta})(\xi)| \d\xi}_{=II}\\
&+ \underbrace{\int_{A_j^{\omega}\cap [-\delta^{-1},\delta^{-1}]^c} E_j(\xi) |\F(\chi_{\delta})(\xi)| \d\xi}_{=III} + \underbrace{\int_{(A_j^{\omega})^c \cap [-\delta^{-1},\delta^{-1}]^c} E_j(\xi) |\F(\chi_{\delta})(\xi)| \d\xi}_{=IV}. 
\end{align*}

We estimate $I,II,III$ and $IV$.

For $|\xi|\leq \delta^{-1}$ we have $|\F(\chi_{\delta})(\xi)| \ll 1$. Each $A_{j,l}^{\omega}$ is an interval of length $2\frac{\omega}{1+\epsilon} k_j^{-\epsilon}$ with midpoint $\frac{l}{1+\epsilon}k_j^{-\epsilon}$, it follows that
\begin{equation*}
m(A_j^{\omega}\cap [-\delta^{-1},\delta^{-1}]) \ll \frac{(1+\epsilon)k_j^{\epsilon}}{\delta} \cdot  \frac{\omega}{1+\epsilon} k_j^{-\epsilon} \ll \frac{\omega}{\delta}.
\end{equation*}
So $I\ll \frac{j^{a-1} \omega}{\delta}$.

Considering $III$, we have
\begin{align*}
\frac{III}{j^{a-1}} & \ll \sum_{|l|\geq \ffloor{(1+\epsilon)k_j^{\epsilon}}{\delta}} \int_{A_{j,l}^{\omega}} \frac{1}{|\xi|^2\delta^2} \d\xi \ll \frac{1}{\delta^2} \sum_{|l|\geq \ffloor{(1+\epsilon)k_j^{\epsilon}}{\delta}} \int_{\frac{l-\omega}{(1+\epsilon)k_j^{\epsilon}}}^{\frac{l+\omega}{(1+\epsilon)k_j^{\epsilon}}} \frac{1}{|\xi|^2}\\
&\ll \frac{1}{\delta^2} \sum_{|l|\geq \ffloor{(1+\epsilon)k_j^{\epsilon}}{\delta}}  \left(\frac{(1+\epsilon)k_j^{\epsilon}}{l-\omega}- \frac{(1+\epsilon)k_j^{\epsilon}}{l+\omega} \right)
\ll \frac{1}{\delta^2} \sum_{|l|\geq \ffloor{(1+\epsilon)k_j^{\epsilon}}{\delta}}   \frac{\omega(1+\epsilon)k_j^{\epsilon}}{l^2} \ll \frac{\omega}{\delta}.
\end{align*}
Thus $III\ll \frac{j^{a-1} \omega}{\delta}$.

In order to estimate $II$, first we write 
\begin{equation*}
(A_j^{\omega})^c= \bigcup_{l\in \Z} \underbrace{\left( \frac{l+\omega}{1+\epsilon} k_j^{-\epsilon}, \frac{l+1-\omega}{1+\epsilon} k_j^{-\epsilon}\right)}_{=:C_{j,l}^{\omega}}.
\end{equation*}
So we can rewrite
\begin{equation*}
II  \ll \int_{(A_j^{\omega})^c\cap [-\delta^{-1},\delta^{-1}]} \frac{1}{\{\xi (1+\epsilon)k_j^{\epsilon}\}} |\F(\chi_{\delta})(\xi)| \d\xi
\ll \sum_{|l|\leq \fceil{(1+\epsilon) k_j^{\epsilon}}{\delta}} \int_{C_{j,l}^{\omega}} \frac{1}{\{\xi (1+\epsilon)k_j^{\epsilon}\}}\d\xi.
\end{equation*}

Note that, for each $l\in\Z$,
\begin{equation}\label{Cljepssum}
\begin{aligned}
\int_{C_{j,l}^{\omega}}& \frac{1}{\{\xi (1+\epsilon)k_j^{\epsilon}\}}\d\xi \ll \frac{1}{(1+\epsilon)k_j^{\epsilon}} \int_{\omega}^{1-\omega} \frac{1}{\min(x,1-x)} \d x \\
&\ll \frac{1}{(1+\epsilon)k_j^{\epsilon}}|\log(\omega)|.
\end{aligned}
\end{equation}
Hence $II\ll \frac{|\log(\omega)|}{\delta}$. 

Lastly for $IV$ we have
\begin{align*}
IV& \ll \frac{1}{\delta^2}\sum_{|l|\geq \ffloor{(1+\epsilon) k_j^{\epsilon}}{\delta}} \int_{C_{j,l}^{\omega}} \frac{1}{\{\xi (1+\epsilon)k_j^{\epsilon}\}} \frac{1}{|\xi|^2}\d\xi\\
&\ll \frac{(1+\epsilon)^2 k_j^{2\epsilon}}{\delta^2}\sum_{|l|\geq \ffloor{(1+\epsilon) k_j^{\epsilon}}{\delta}} \frac{1}{|l|^2}\int_{C_{j,l}^{\omega}} \frac{1}{\{\xi (1+\epsilon)k_j^{\epsilon}\}}\d\xi
\end{align*}
using \eqref{Cljepssum} we obtain $IV\ll \frac{|\log(\omega)|}{\delta}$.

It follows that
\begin{equation*}
\int_{\R} E_j\left( \xi \right) |\F(\chi_{\delta})(\xi)|\d\xi \ll j^{a-1} \frac{\omega}{\delta} + \frac{|\log(\omega)|}{\delta},
\end{equation*}
for all $\omega>0$ and $j=1,...,J(N)$. Choosing $\omega=j^{-a+1}$ yields
\begin{equation*}
\int_{\R} S_{B'_{N,j}}(\xi) |\F(\chi_{\delta})(\xi)| \d\xi \ll \int_{\R} E_j\left( \xi \right) |\F(\chi_{\delta})(\xi)|\d\xi \ll \frac{1}{\delta} \log(j). 
\end{equation*}

\end{proof}

\begin{proposition}\label{1+epsprop}
If $\phi$ is $\rho'$-Hölder continuous has polynomially decaying twisted ergodic integrals $x\in X$ 
with rate $\gamma_{f,x}(T)\ll ||f||_{\B} T^{-\kappa}$ and assume 
\begin{equation*}
2\epsilon_i-\kappa \epsilon_i<2d-2+\kappa, \; (1+\epsilon_i)(1-\kappa)<d \; \textit{ and } \; \epsilon_i<1
\end{equation*}
for all $i=1,...,d$.
Then, for some $\kappa'>0$ we have
\begin{equation*}
\frac{1}{N^d}\left| \sum_{\textbf{b}\in B_N} f(\phi_{\textbf{b}}(x)) \right|\ll ||f||_{\B} N^{\kappa'},
\end{equation*} 
for all $N\geq 1$
\end{proposition}
\begin{proof}
Let $\delta>0$. Analogously to the above let $a_1,...,a_d>0$ and $k_{j,i}=\lceil j^{a_i} \rceil$ for $i=1,...,d$ and $j=1,...,J_i(N)$ where $J_i(N)=\lceil N^{\frac{1}{a_i}} \rceil$. Let
\begin{equation*}
b'_{n,i}=k_{j,i}^{1+\epsilon}+ (1+\epsilon) k_{j,i}^{\epsilon} (n-k_{j,i}) , \;\;\; \textit{ for } k_{j,i}\leq n <k_{j+1,i}, j\geq 1,
\end{equation*}
and $B'_{N,j,i}=\{b'_n \;|\; k_{j-1,i} \leq n \leq \min(N,k_{j,i})-1\}$. Let
\begin{equation*}
\tilde{B}_{N,j_1,...,j_d}=\prod_{i=1}^d B'_{N,j_i,i}, \; \textit{ and } \; \tilde{B}_N=\bigcup_{j_1,...,j_d} \tilde{B}_{N,j_1,...,j_d}.
\end{equation*}
Then we have
\begin{align*}
\left| \sum_{\textbf{b}\in B_N} f(\phi_{\textbf{b}}(x)) \right. & \left. - \sum_{\textbf{b}'\in \tilde{B}_N} f(\phi_{\textbf{b'}}(x))\right|\\
&\ll ||f||_{\B} \sum_{\substack{j_i=1,...,J_i(N)\\i=1,...,d}} \sum_{\substack{l_i=0,...\min(N,k_{j,i})-k_{j-1,i}-1 \\i=1,...,d}} \sum_{i=1}^{d} j_i^{a_i(\epsilon_i-1)} l_i^2\\
&\ll N^{d-1} \sum_{i=1}^{d} N^{2+\epsilon_i-\frac{2}{a_i}}.
\end{align*}
Furthermore
\begin{equation*}
d(\tilde{B}_{N,j_1,...,j_d}) \ll \sum_{i=1}^d j_i^{a_i(1+\epsilon_i)-1}.
\end{equation*}
Hence, using Proposition \ref{sparseequilem1point} and Lemma \ref{b'lem}, we have
\begin{align*}
&\left| \sum_{\textbf{b}'\in \tilde{B}_N} f(\phi_{\textbf{b}'}(x))\right| \ll \sum_{j_1,...,j_d} \left| \sum_{\textbf{b}'\in \tilde{B}_{N,j_1,...,j_d}} f(\phi_{\textbf{b}'}(x)) \right|\\
&\ll \sum_{j_1,...,j_d} \left( 1+ \# \tilde{B}_{N,j_1,...,j_d} \delta^{\rho \rho'} ||f||_{\B}\right)\\
&\;\;\;+ \sum_{j_1,...,j_d} d(\tilde{B}_{N,j_1,...,j_d}) \gamma_{f,x}(d(\tilde{B}_{N,j_1,...,j_d})) \int_{\R^d} S_{\tilde{B}_{N,j_1,...,j_d}} \left(\xi \right) \prod_{j=1}^d |\F(\chi_{\delta})(\xi_j)| \d\xi \\
&\ll N^{\sum_{i=1}^d \frac{1}{a_i}} + N^d \delta^{\rho \rho'} ||f||_{\B} + ||f||_{\B}\sum_{j_1,...,j_d} \sum_{i=1}^d j_i^{(a_i(1+\epsilon_i)-1)(1-\kappa)} \prod_{i=1}^d \frac{1}{\delta} \log(j_i)\\
&\ll N^{\sum_{i=1}^d \frac{1}{a_i}} + N^d \delta^{\rho \rho'} ||f||_{\B} + ||f||_{\B} \frac{1}{\delta^d} \log^d(N) \sum_{j_1,...,j_d} \sum_{i=1}^d j_i^{(a_i(1+\epsilon_i)-1)(1-\kappa)} \\
&\ll N^{\sum_{i=1}^d \frac{1}{a_i}} + N^d \delta^{\rho \rho'} ||f||_{\B} + ||f||_{\B} \frac{1}{\delta^d} \log^d(N)  \sum_{i=1}^d N^{(1+\epsilon_i)(1-\kappa)+\frac{\kappa}{a_i}}.
\end{align*}

Setting 
\begin{equation*}
\delta=\left(\frac{N^d}{ \sum_{i=1}^d N^{(1+\epsilon_i)(1-\kappa)+\frac{\kappa}{a_i}}}\right)^{-\frac{1}{d+ \rho \rho'}}
\end{equation*}
yields
\begin{align*}
&\left| \sum_{\textbf{b}'\in \tilde{B}_N} f(\phi_{\textbf{b}'}(x))\right|\\
&\ll N^{\sum_{i=1}^d \frac{1}{a_i}}+ N^d ||f||_{\B}\log^d(N) \left(\frac{N^{d}}{\sum_{i=1}^d N^{(1+\epsilon_i)(1-\kappa)+\frac{\kappa}{a_i}}}\right)^{-\frac{1}{d+ \rho \rho'}}.
\end{align*}
Hence
\begin{equation}\label{1+epseq}
\begin{aligned}
& \left| \sum_{\textbf{b}\in B_N} f(\phi_{\textbf{b}}(x)) \right|\\
& \ll N^{d-1} \sum_{i=1}^{d} N^{2+\epsilon_i-\frac{2}{a_i}} + N^{\sum_{i=1}^d \frac{1}{a_i}}+ N^d ||f||_{\B}\log^d(N) \left(\frac{N^{d}}{\sum_{i=1}^d N^{(1+\epsilon_i)(1-\kappa)+\frac{\kappa}{a_i}}}\right)^{-\frac{1}{d +\rho \rho'}}\\
&  \ll ||f||_{\B}\log^d(N) (N^{d+1 +\max_i \left(\epsilon_i-\frac{2}{a_i}\right)}+ N^{d \max_i\left( \frac{1}{a_i}\right)} + N^{d-\frac{d}{d+\rho \rho'} + \frac{1}{d+\rho \rho'}\max_i \left((1+\epsilon_i)(1-\kappa)+\frac{\kappa}{a_i}\right)}.
\end{aligned}
\end{equation}

The expression in \eqref{1+epseq} is $o(N^d)$ (as desired) if the $a_i$ can be chosen such that all of the exponents are $<1$. In particular, this is satisfied if we can choose $a_i$ such that

\begin{itemize}
\item \begin{equation*}
\epsilon_i-\frac{2}{a_i} <-1 \;\;\; \forall i=1,...,d,
\end{equation*}
\item \begin{equation*}
a_i>1 \;\;\; \forall i=1,...,d
\end{equation*}
\item and \begin{equation*}
 (1+\epsilon_i)(1-\kappa)+\frac{\kappa}{a_i}<d \;\;\; \forall i=1,...,d.
\end{equation*}
\end{itemize}

 Rewriting, we can easily deduce that these conditions are satisfied if $(1+\epsilon_i)(1-\kappa)<d$ and 
\begin{equation*}
\min\left(1,\frac{\kappa}{d-(1+\epsilon_i)(1-\kappa)}\right)< a_i < \frac{2}{1+\epsilon_i}.
\end{equation*}

Such $a_i$ exist, whenever
\begin{equation*}
\frac{\kappa}{d-(1+\epsilon_i)(1-\kappa)} <  \frac{2}{1+\epsilon_i}
\end{equation*}
or equivalently
\begin{equation*}
2\epsilon_i-\kappa \epsilon_i<2d-2+\kappa.
\end{equation*}
\end{proof}

\section{Examples}\label{explesec}

\subsection{Time change}
Many of our examples are smooth time changes of well-known flows. We will verify polynomial (weak) mixing and polynomial ergodicity for each of them. For polynomial ergodicity, it is enough to verify it for the original flow. Indeed, as we will show below, polynomial ergodicity of a smooth flow is equivalent to polynomial ergodicity of any smooth time change.

Let $\phi$ be a $C^r$ flow on a compact manifold $X$. 
Let $\tau:X\rightarrow (0,\infty)$ be a positive smooth function bounded away from $0$ and $\infty$, i.e $0< \inf \tau \leq \sup \tau < \infty$, and, for $t\in \R$, let $\sigma(t,x)$ be defined by
\begin{equation}\label{sigmadef}
t=\int_0^{\sigma(t,x)} \tau(\phi_s(x)) \d s \;\;\;x\in X.
\end{equation}
Note that $\sigma$ is a cocycle, i.e. for $t,t'\in \R, x\in X$ we have
\begin{equation*}
\sigma(t+t',x)=\sigma(t,x)+\sigma(t',\phi_{t}(x)).
\end{equation*}
Therefore, the \textit{time-change} $\phi^{\tau}$ given by
\begin{equation*}
\phi^{\tau}_t(x)=\phi_{\sigma(t,x)}(x) \;\;\; x\in X, t\in \R,
\end{equation*}
is well-defined. If $\phi$ is the flow satisfying
$
\frac{\d\phi}{\d t}=V,
$
for a vectorfield $V$, then $\phi^{\tau}$ satisfies
$
\frac{\d\phi^{\tau}}{\d t}=\tau V.
$

Suppose $\phi$ preserves the measure $\mu$ and $\tau\in C^1$, then for $f\in C^1$
\begin{equation*}
\frac{\d}{\d t} \int_X f(\phi^{\tau}_t(x)) \tau(x) \d\mu(x) = \frac{\d}{\d t} \int_X f(\phi_{\sigma(t,x)}(x)) \tau(x) \d\mu(x) = \frac{\d}{\d s} \int_X f(\phi_s(x)) \d\mu(x) =0,
\end{equation*}
hence $\phi^{\tau}$ preserves the measure $\mu^{\tau}$ with density $\tau$, i.e $\frac{\d \mu^{\tau}}{\d\mu}=\tau$. For simplicity assume $\int_X\tau \d\mu=1$.\\

\begin{proposition}\label{polyergprop}
If $\phi$ is polynomially ergodic, then so is $\phi^{\tau}$.
\end{proposition}
\begin{proof}
Fix $x\in X$ and denote $\sigma(t)=\sigma(t,x)$. From \eqref{sigmadef} we obtain
\begin{equation*}
\tau(\phi_{\sigma(t)}(x)) \sigma'(t)=1\;\;\; t\in \R,
\end{equation*}
hence for $f\in C^r\cap L^2_0 (\mu^{\tau})$ we have
\begin{equation*}
\left|\int_0^T f(\phi_{\sigma(t)}(x)) \d t\right|= \left|\int_0^{\sigma(T)} \frac{1}{\sigma'(\sigma^{-1}(s))} f(\phi_s(x)) \d s\right|  = \left|\int_0^{\sigma(T)} \tau(\phi_s(x)) f(\phi_s(x)) \d s\right|.
\end{equation*}
Since $\sigma(T) \leq \sup \tau T$ and $\int \tau f \d\mu= \int f \d\mu^{\tau}=0$, we have 
\begin{equation*}
\left|\int_0^T f(\phi_{\sigma(t)}(x)) \d t\right|\ll ||f||_{C^r} T^{1-\delta},
\end{equation*}
for some $\delta>0$.
\end{proof}

\subsection{Unipotent flows}

Let $G$ be a semisimple Lie group with finite centre and no compact factors and let $\Gamma < G$ be a cocompact, irreducible lattice, denote $X= \Gamma / G$. For $V\in \mathfrak{g}$ consider the flow
\begin{equation*}
\psi^V_t(\Gamma g)=\Gamma g exp(V t) , \; x=\Gamma g\in X.
\end{equation*}
Let $u\in \mathfrak{g}\setminus \{0\}$ be a \textit{nilpotent} element, i.e $[u,\cdot]$ is a (non-zero) nilpotent operator on $\mathfrak{g}$, then $\phi=\psi^u$ is called the \textit{unipotent flow}. Note that clearly $||\psi^u_t||_{C^r}$ grows polynomially in $t$.

By the Jakobson-Morozov Theorem, there is a subalgebra of $\mathfrak{g}$ containing $u$ and isomorphic to $\mathfrak{sl}(2,\R)$. In particular, there is an element $V\in \mathfrak{g}$ contained in this subalgebra such that the following renormalisation holds
\begin{equation*}
\psi^u_t \circ \psi^V_s= \psi^V_s\circ \psi^u_{t e^s} \;\;\;\forall t,s \in \R,
\end{equation*}
$\psi^V$ is called the geodesic flow.

Polynomial mixing is well known and true in great generality for unipotent flows and their time changes. This is shown in \cite[Theorem 5]{ravottiunimix}; Let $\tau$ be a smooth time change then there is an $\delta>0$ and $r\geq 1$ such that
\begin{equation*}
|\langle f\circ \phi^{\tau}_t, g \rangle | \ll ||f||_{C^r} ||g||_{C^r} T^{-\delta}, \;\;\;\forall f,g \in C^r\cap L^2_0(\mu^{\tau}),
\end{equation*}
where $\mu$ is the Haar measure.

On the other hand, polynomial decay of ergodic averages (at certain points) is much harder to obtain, and it is an active area of research. The references below prove results of the following type: For all $T>0$ and $x\in X$ either
\begin{itemize}
\item[(i)]
\begin{equation}\label{sleq}
\frac{1}{T}\left|\int_0^T f(\phi_t(x))\d t\right|\ll T^{-\kappa} ||f||_{C^r} \;\;\;\forall f\in C^r\cap L^2_0,
\end{equation}
\item[(ii)] or $x$ is in some exceptional set, $x\in X_T$.
\end{itemize} 

Some examples of actions with explicit control on ergodic averages are.

\begin{itemize}
\item The action of
\begin{equation*}
\begin{bmatrix}
1&t\\
0&1
\end{bmatrix}
\end{equation*}
on compact space $\Gamma / Sl(2,\R)$. In this case the polynomial equidistribution
of the flow orbits is classical. In fact, also
polynomial equidistribution of the time 1 map
is known, see \cite{tchorofu}, but our methods provide an alternative proof. 

\item The action of
\begin{equation*}
\begin{bmatrix}
1&t\\
0&1
\end{bmatrix} \times \begin{bmatrix}
1&t\\
0&1
\end{bmatrix}
\end{equation*}
on compact space $\Gamma / Sl(2,\R) \times Sl(2,\R)$ (\cite[Theorem 1.1]{lindenstrauss2022effective}). Here, the obstruction to being equidistributed is starting close to a periodic orbit. However, if $X$ is compact, there are no periodic orbits\footnote{Let $r_0$ be the injectivity radius of $X$, then every periodic orbit has period $\geq r_0$ (making $r_0$ smaller if necessary). Suppose $x$ is a periodic point, say of period $p$, since
\begin{equation*}
a_t = \begin{bmatrix}
e^{-t}&0\\
0&e^{t}
\end{bmatrix} \times \begin{bmatrix}
e^{-t}&0\\
0&e^{t}
\end{bmatrix}
\end{equation*}
renormalises the flow, $x a_{-t}$ has period $p e^{-2t}$. Letting $t\rightarrow\infty$ leads to a contradiction, hence there are no periodic points.}.
\item The action of
\begin{equation*}
\begin{bmatrix}
1&0&0\\
0&1&t\\
0&0&1
\end{bmatrix}
\end{equation*}
on $Sl(3,\Z) / Sl(3,\R)$ (\cite[Theorem 1.2]{yang2023effective}).  Here the (polynomial) equidistribution is not uniform in $X$, however, it is proven that for all points $x$ and big enough $T$ we have\footnote{More explicitly it is shown that $(x,T)$ satisfies \eqref{sleq}, or $x$ is in some exceptional set $X_T$. However, as pointed out in the references, Mahler's criterion implies that $X_T\cap K =\emptyset$ for each compact set $K\subset Sl(3,\Z) / Sl(3,\R)$ and big enough $T$.}
\begin{equation*}
\frac{1}{T}\left|\int_0^T f(\phi_t(x))\d t\right|\ll T^{-\kappa} ||f||_{C^r} \;\;\;\forall f\in C^r\cap L^2_0.
\end{equation*}

\end{itemize}

\subsection{Heisenberg nilflows}
Let $N$ be the Heisenberg group
\begin{equation*}
N=\left\{ \begin{bmatrix}
1&a&c\\
0&1&b\\
0&0&1
\end{bmatrix} \;|\; a,b,c\in \R\right\}
\end{equation*}
its Lie algebra $\mathfrak{n}$ is generated (as a vector space) by
\begin{equation*}
A=\begin{bmatrix}
0&1&0\\
0&0&0\\
0&0&0
\end{bmatrix}, \quad B=\begin{bmatrix}
0&0&0\\
0&0&1\\
0&0&0
\end{bmatrix}, \quad\text{and}\quad C=\begin{bmatrix}
0&0&1\\
0&0&0\\
0&0&0
\end{bmatrix}.
\end{equation*}
Let $\Gamma < N$ be a co-compact lattice, it is 
well-known (see \cite{auslander1977lecture}) that, up to automorphism of $N$, there is a positive integer $E\in \N$ such that
\begin{equation*}
\Gamma = \left\{ \begin{bmatrix}
1&n&\frac{k}{E}\\
0&1&m\\
0&0&1
\end{bmatrix} \;|\; n,m,k \in \Z \right\}.
\end{equation*}
Denote $X=\Gamma / N$, and let $W=w_aA+w_b B+w_c C\in \mathfrak{n}$, for $w_a,w_b,w_c\in \R$. The \textit{Heisenberg nilflow} $\phi_t:X \rightarrow X$ is given by $\phi_t(x) = x \exp(tW) \;\;\;x\in X.$
More explicitly
\begin{equation*}
\phi_t\left( \Gamma \begin{bmatrix}
1&a&c\\
0&1&b\\
0&0&1
\end{bmatrix} \right) = \Gamma \begin{bmatrix}
1&a+t w_a & c+tx_c+t a w_b + \frac{t^2}{2} w_a w_b\\
0&1&b+tw_b\\
0&0&1
\end{bmatrix}.
\end{equation*}
In order to show polynomial ergodicity for $\phi$, we will conjugate it to a special flow over a map on the torus $\T^2$. To this end, first note that $\T^2$ is embedded in $X$ by
\begin{equation*}
\pi:(x,y) \in \T^2_E=\R^2 \backslash (\Z \times \Z \backslash E\Z) \mapsto \Gamma exp(xA+yC)=\Gamma \begin{bmatrix}
1&x&y\\
0&1&0\\
0&0&1
\end{bmatrix}
\end{equation*}
then $\pi$ is well-defined and a smooth homomorphism. So consider the section $Y=\pi(\T^2_E)\subset X$, the first return time to $Y$ is constant equal to $r=\frac{1}{w_b}$ and the first return map $R$ is given by
\begin{equation*}
R(\pi(x,y))= \pi\left( x+ \frac{w_a}{w_b},x+y+\frac{w_c}{w_b}+\frac{w_a}{2w_b} \right).
\end{equation*}

So in order to prove polynomial ergodicity of a Heisenberg nilflow, it is enough to prove polynomial ergodicity of a skew shift on $\T^2$, say of 
\begin{equation*}
R=R_{\eta_1,\eta_2}:(x,y) \mapsto (x+\eta_1, x+y+\eta_2),
\end{equation*}
where $\eta_1,\eta_2\in[0,1)$.\\
For $f\in H^2_0(\T^2)$ write $f(\textbf{x})=\sum_{\textbf{k}\in \Z^2} a_{\textbf{k}} e(\langle \textbf{k}, \textbf{x} \rangle)$ and first assume $a_{k_1,0}\equiv 0$. Then, for $\textbf{x}\in \T^2$ and $N\geq 1$, we have
\begin{align*}
\left| \sum_{n=0}^{N-1} f(R^n(\textbf{x})) \right| & = \left| \sum_{n=0}^{N-1} \sum_{\textbf{k}\in \Z^2} a_{\textbf{k}} e(\langle \textbf{k}+(0,n), \textbf{x} \rangle + n k_1\eta_1 +nk_2\eta_2+{n \choose 2}k_2\eta_1) \right|\\
&\leq \sum_{\textbf{k}\in \Z^2} |a_{\textbf{k}}| \left| \sum_{n=0}^{N-1} e(\langle \textbf{k}+(0,n), \textbf{x} \rangle + n k_1\eta_1 +nk_2\eta_2+{n \choose 2}k_2\eta_1) \right|.
\end{align*}
By \cite[Lemma 3.1]{matomaekipoly} (see also \cite[Propostition 4.3]{taopoly}) we have, for each $\textbf{k}\in \Z^2$, 
\begin{equation*}
\left| \sum_{n=0}^{N-1} e(\langle \textbf{k}+(0,n), \textbf{x} \rangle + n k_1\eta_1 +nk_2\eta_2+{n \choose 2}k_2\eta_1) \right| \ll N^{1-\delta},
\end{equation*}
for some $\delta>0$, provided $\eta_1$ is diophantine. In that case
\begin{equation*}
\left| \sum_{n=0}^{N-1} f(R^n(\textbf{x})) \right| \ll ||f||_{H^2} N^{1-\delta}.
\end{equation*}
Now for general $f\in H^r_0$, for $r$ big enough, we split into $f=f_1+f_2$, where
\begin{equation*}
f_1=\sum_{k\in \mathbb{Z}^2,k_2\not =0} a_k e_k \;\;\textit{ and }\;\;f_2=\sum_{k\in \mathbb{Z}^2,k_2=0} a_k e_k.
\end{equation*}
For $f_1$ we just showed
\begin{equation*}
\left| \sum_{n=0}^{N-1} f_1(R^n(\textbf{x})) \right| \ll ||f_1||_{H^r} N^{1-\delta},
\end{equation*}
while $f_2$ is a coboundary over $R$. Hence, $R$ is polynomially ergodic.

It follows that $\phi$ is also polynomially ergodic and, by Propostition \ref{polyergprop}, so is any smooth time change $\phi^{\tau}$.

Turning our attention to weak mixing, first note that $\phi$ itself can never be weak mixing, since it is a constant roof suspension. However, \cite[Theorem 2.4]{forni2017timechanges} shows that generic time changes are polynomially weak mixing.

\subsection{Higher dimensional horospheres on compact manifolds}\label{ddimhorosec}

For $m_1,m_2\geq 1$ consider the subgroup 
\begin{equation*}
H=\left\{ \begin{bmatrix}
Id_{m_1} & Y \\ 0 & Id_{m_2}
\end{bmatrix}, \; Y\in M_{m_1\times m_2} \right\} < Sl(m_1+m_2, \R),
\end{equation*}
and let $\Gamma<Sl(m_1+m_2,\R)$ be a cocompact lattice. Then $H$ is abelian and hence induces an $\R^d$-action, where $d=m_1m_2$, on $\Gamma / Sl(m_1+m_2,\R)$. Denote this action, with the obvious identification\footnote{Identify $\iota(t_1,...,t_d)=\begin{bmatrix}
Id_{m_1} & Y \\ 0 & Id_{m_2}
\end{bmatrix}$, where $Y=(t_{m_2(i-1)+j)})_{1\leq i\leq m_1, 1\leq j\leq m_2}$} $\iota$ of $H$ with $\R^d$, by $\phi$. The group $H$ is expanded by the (partially) hyperbolic element
\begin{equation*}
g_s=diag(e^{-\frac{s}{m_1}},...,e^{-\frac{s}{m_1}},e^{\frac{s}{m_2}},...,e^{\frac{s}{m_2}}).
\end{equation*}
Explicitly
\begin{equation*}
g_{-s} \begin{bmatrix}
Id_{m_1} & Y \\ 0 & Id_{m_2}
\end{bmatrix} g_{s} = \begin{bmatrix}
Id_{m_1} & e^{\rho s} Y \\ 0 & Id_{m_2}
\end{bmatrix},
\end{equation*}
where $\rho=\frac{1}{m_1}+\frac{1}{m_2}$, after identification $g_{-s} \iota(\textbf{t}) g_s= \iota(e^{\rho s} \textbf{t})$. 

In order to apply our methods and show pointwise sparse equidistribution, we need to first show that the action $\phi$ is polynomially mixing and has decaying ergodic averages of polynomial speed (in particular, it is uniquely ergodic for the Haar-measure $\mu=\mu_{Sl(m_1+m_2)}$). Both of these facts are shown in \cite{KMbdd}.

First, polynomial mixing follows from \cite[Corollary 2.4.4]{KMbdd} where it is shown that
\begin{equation*}
|\langle f_1\circ h, f_2 \rangle_{L^2(\mu)} | \ll e^{-\lambda dist(h,Id)} ||f_1||_{C^r} ||f_2||_{C^r} \;\;\; \forall f_1,f_2\in C^r\cap L^2_0(\mu),
\end{equation*}
for some $r\geq 1,\lambda>0$ and a Riemannian metric $dist(\cdot,\cdot)$. In particular, $\phi$ is polynomially weak mixing.

Turning our attention to ergodic averages, \cite[Proposition 2.4.8]{KMbdd} shows that
\begin{equation*}
\left| \int_{\R^d} f(\textbf{t}) \psi(g_s \iota(\textbf{t}) x) \d\textbf{t} \right|  \ll e^{-\lambda' s} ||f||_{H^r(\R^d)} ||\psi||_{C^r} 
\end{equation*}
for all $f\in C^r_c(\R^d),\psi \in C^r\cap L^2_0(\mu), x\in Sl(m_1+m_2,\R) \backslash \Gamma, s>0,$
and some $\lambda'>0$. Let $f\in C_c^r(\R^d)$ be such that $||f-1_{[0,1]^d}||_{L^{1}(\R^d)}\ll 1$ and $||f||_{H^r}\ll 1$. For $x_0\in Sl(m_1+m_2,\R) \backslash \Gamma$, $\psi \in C^r\cap L^2_0(\mu)$ and $T>1$, denote $s=\frac{1}{\rho} \log(T)$, we have 
\begin{align*}
\left| \int_{[-T,T]^d} \psi(\iota(\textbf{t})) \d\textbf{t} \right| & \ll \left| \int_{\R^d} f(\textbf{t}) \psi(g_s \iota(T \textbf{t}) (g_{-s} x_0)) \d\textbf{t} \right| + ||\psi||_{C^r}\\
&\ll (2T)^d \left| \int_{\R^d} f(T^{-1} \textbf{t}) \psi(g_s \iota(\textbf{t}) (g_{-s} x_0)) \d\textbf{t} \right| + ||\psi||_{C^r}\\
&\ll T^{d-\rho^{-1}\lambda'} ||\psi||_{C^r}.
\end{align*}

Now we may apply Theorem \ref{explethm}, to deduce equidistribution of $\phi$ at all points along\footnote{Rather, on sets of the form $\{(n_1^{1+\epsilon_1},...,n_d^{1+\epsilon_d})\}$ as in Definition \ref{1+epsdef}.} $n^{1+\epsilon}$, the rates decay polynomially.

\subsection{Linear flows}\label{linsec}

Let $\alpha\in (0,1)^d$ be \textit{diophantine}, i.e. there is a $n \geq 1$ such that
\begin{equation*}
|\textbf{k}|^{-n} \ll |\langle \textbf{k}, \alpha \rangle -l| \;\;\;\forall \textbf{k}\in \Z^{d-1} \setminus \{\textbf{0}\}, l\in \Z.
\end{equation*} 
Let $\phi_t(x)=x+\alpha$ be the linear flow on $\T^d$ with angle $\alpha$. \\
In order to complete the proof of Remark \ref{linrem}, we shall, for all $\textbf{k}\in \Z^d$, provide a polynomial rate of decay for 
\begin{equation*}
\frac{1}{N} \sum_{n=0}^{N-1} e(\xi n^p) \rightarrow 0 \;\;\;\as{N}
\end{equation*}
for $p>0$ and $\xi =\langle \textbf{k}, \alpha\rangle \not \in \Z$. An explicit bound shall be given in Corollary \ref{lincor}.

Our arguments will be based on a quantitative version of Féjer's Theorem (see \cite[Theorem 2.4]{kuipers1974uniform}) and van der Corput's Lemma. For convenience of the reader, let us reprove the quantitative versions.

\begin{lemma} 
Let $f:[0,\infty) \rightarrow\R$ be $C^1$ such that $f'(x)$ is eventually monotone. Then it holds that
\begin{equation} \label{fejerlemclaim}
\frac{1}{N} \left| \sum_{n=0}^{N-1} e(f(n)) \right| \ll \frac{1}{N} \left( \frac{1}{|f'(N)|} + |f(N)| \right),
\end{equation}
for big enough $N$.
\end{lemma}
\begin{proof}
For each $n\geq 0$, denote $\Delta f(n)=f(n+1) - f(n)$, by the Mean Value Theorem it holds that
\begin{align*}
\left| e(f(n+1)) - e(f(n)) -2\pi i \Delta f(n) e(f(n)) \right| \ll (\Delta f(n))^2,
\end{align*}
dividing by $\Delta f (n)$
\begin{equation*}
\left| \frac{e(f(n+1))}{\Delta f (n+1)} - \frac{e(f(n))}{\Delta f(n)} -2\pi i e(f(n)) \right| \ll |\Delta f(n)| + \left|\frac{1}{\Delta f (n)} - \frac{1}{\Delta f (n+1)}\right|.
\end{equation*}
Using this, we have
\begin{align*}
\left| 2\pi i \sum_{n=0}^{N-1} e(f(n)) \right| &\ll \left| \sum_{n=0}^{N-1} \left( -\frac{e(f(n+1))}{\Delta f (n+1)} + \frac{e(f(n))}{\Delta f(n)} +2\pi i e(f(n))\right) + \frac{e(f(N))}{\Delta f (N)} - \frac{e(f(0))}{\Delta f(0)} \right|\\
&\ll \sum_{n=0}^{N-1} \left|\frac{1}{\Delta f (n)} - \frac{1}{\Delta f (n+1)}\right| + \frac{1}{|\Delta f (N)|} + \frac{1}{|\Delta f (0)|} + \sum_{n=0}^{N-1} \Delta f(n),
\end{align*}
and the claim follows since $\Delta f (n)$ is monotone.
\end{proof}

The following is Lemma 3.1 of \cite{kuipers1974uniform}.

\begin{lemma}\label{Corputlem}
Let $\s{u}{n}$ be a sequence of complex numbers, then, for every $1\leq H \leq N$ it holds that
\begin{equation*}
H^2 \left| \sum_{n=1}^N u_n \right|^2 \leq H(N+H-1) \sum_{n=1}^N |u_n|^2 + 2(N+H-1) \sum_{h=1}^{H-1} (H-h) \text{Re} \sum_{n=1}^{N-h}  u_n \bar{u}_{n+h}.
\end{equation*}
\end{lemma}

Setting $u_n=e(f(n-1))$ in Lemma \ref{Corputlem}, for some monotone $f:\N \rightarrow\R$, and dividing by $N^2H^2$ we obtain
\begin{equation}\label{Corputexp}
\begin{aligned}
\frac{1}{N^2} &\left| \sum_{n=0}^{N-1} e(f(n)) \right|^2\\
&\leq \frac{N+H-1}{NH}+ 2\sum_{h=1}^{H-1} \frac{(N+H-1)(H-h)(N-h)}{N^2H^2} \left| \frac{1}{N-h} \sum_{n=0}^{N-h-1} e(f(n+h)-f(n)) \right|\\
& \ll \frac{1}{H} + \frac{1}{H} \sum_{h=1}^{H-1} \left| \frac{1}{N-h+1} \sum_{n=0}^{N-h-1} e(f(n+h)-f(n)) \right|
\end{aligned}
\end{equation}

\begin{proposition}
Let $k\geq 1$ and $f:[0,\infty) \rightarrow\R$ be $C^k$ such that $f^{(k)}$ is eventually monotone. Then it holds that
\begin{equation}\label{Corputfk}
\left| \frac{1}{N} \sum_{n=0}^{N-1} e(f(n)) \right|^{2(k-1)} \ll \frac{1}{N} \left( \frac{1}{|f^{(k)}(N)|} + |f^{(k-1)}(N)| \right)
\end{equation}
for big enough $N$.
\end{proposition}
\begin{proof}
For $k=1$ this claim coincides with \eqref{fejerlemclaim}. Now per induction, assume the claim holds for some $k\geq 1$, we will show that it also holds for $k+1$.

Indeed, for $h\geq 1$ denote $g_h(x)=f(x+h)-f(x)$, then, by induction hypothesis, it holds that
\begin{equation*}
\left| \frac{1}{N} \sum_{n=0}^{N-1} e(g_h(n)) \right|^{2(k-1)} \ll \frac{1}{N} \left( \frac{1}{|g_h^{(k)}(N)|} + |g_h^{(k-1)}(N)| \right).
\end{equation*}
Using \eqref{Corputexp}, we obtain
\begin{align*}
\left| \frac{1}{N} \sum_{n=0}^{N-1} e(f(n)) \right|^{2k}& \ll \left(\frac{1}{H} + \frac{1}{H} \sum_{h=1}^{H-1} \left| \frac{1}{N-h+1} \sum_{n=0}^{N-h+1} e(g_h(n)) \right|\right)^{2(k-1)}\\
&\ll \frac{1}{H} + \frac{1}{H} \sum_{h=1}^{H-1} \left| \frac{1}{N-h+1} \sum_{n=0}^{N-h+1} e(g_h(n)) \right|^{2(k-1)}\\
& \ll \frac{1}{H} +\frac{1}{NH} \sum_{h=1}^{H-1} \frac{1}{|f^{(k)}(N+h)-f^{(k)}(N)|} &\\
& + \frac{1}{NH} \sum_{h=1}^{H-1} |f^{(k-1)}(N+h)-f^{(k-1)}(N)|,
\end{align*}
where, in the second inequality we used Cauchy-Schwartz, more explicitly $\left( \sum_{j=1}^J a_j \right)^{2k} \leq J^{2k} \sum_{j=1}^J a_j^{2k}$ for $J\geq 1$ and $a_1,...,a_J>0$.
Since $f^{(k)}$ and $f^{(k-1)}$ are eventually monotone, it holds that
\begin{equation*}
\left| \frac{1}{N} \sum_{n=0}^{N-1} e(f(n)) \right|^{2k} \ll \frac{1}{H} +\frac{1}{N} \frac{1}{|f^{(k)}(N)|} + \frac{1}{NH} |f^{(k-1)}(N)|,
\end{equation*}
for big enough $N$. The claim follows with $H=N$.
\end{proof}

\begin{proposition}
Let $p\in \N$ and $f(x)=\sum_{j=0}^p a_j x^j$ be a real polynomial of degree $p$ and leading coefficient $a_p\not \in \Z$, then it holds that
\begin{equation}\label{CorputN}
\left|\frac{1}{N} \sum_{n=0}^{N-1} e(f(n)) \right| \ll \frac{1}{d(\xi , \Z)} N^{-4^{-p+1}} \;\;\; N\geq 1.
\end{equation}
\end{proposition}
\begin{proof}
The proof is a simple induction argument using Lemma \ref{Corputlem}.

First note that, for $p=1$, the sum in \eqref{lincorclaim2} is a geometric sum, hence the claim holds. Now suppose the claim is true for some $p\geq 1$ and let $f$ be a polynomial of degree $p+1$ as in the statement. Using Lemma \eqref{Corputlem} with $u_n=e(f(n))$ and dividing by $N^2H^2$ we obtain
\begin{equation*}
\begin{aligned}
\frac{1}{N^2} &\left| \sum_{n=0}^{N-1} e(f(n)) \right|^2\\
&\leq \frac{N+H-1}{NH}+ 2\sum_{h=1}^{H-1} \frac{(N+H-1)(H-h)(N-h)}{N^2H^2} \left| \frac{1}{N-h} \sum_{n=0}^{N-h-1} e(g_h(n)) \right|\\
& \ll \frac{1}{H} + \frac{1}{H} \sum_{h=1}^{H-1} \left| \frac{1}{N-h+1} \sum_{n=0}^{N-h-1} e(g_h(n)) \right|,
\end{aligned}
\end{equation*}
where $g_h(x)=f(x)-f(x+h)$. Since $g_h$ is a polynomial of degree $p$ and leading coefficient $(p+1)h a_p\not \in \Z$, hence by induction hypotheses
\begin{equation*}
\left| \frac{1}{N-h+1} \sum_{n=0}^{N-h-1} e(g_h(n)) \right| \ll \frac{1}{d((p+1)h a_p,\Z)} N^{4^{-p+1}} \ll \frac{h}{d(a_p,\Z)} N^{4^{-p+1}}.
\end{equation*}
For $H=N^{-2\cdot 4^{-p}}$ we obtain
\begin{align*}
\frac{1}{N^2} &\left| \sum_{n=0}^{N-1} e(f(n)) \right|^2 \ll \frac{1}{d(a_p,\Z)}N^{-2\cdot 4^{-p}}
\end{align*}
and the claim follows.
\end{proof}

\begin{corollary}\label{lincor}
For $p\in (0,\infty) \setminus \N$ and $\xi\in \R\setminus \{0\}$ it holds that 
\begin{equation}\label{lincorclaim1}
\left|\frac{1}{N} \sum_{n=0}^{N-1} e(\xi n^p) \right|^{2\lceil p \rceil} \ll \frac{1}{|\xi|} N^{\lceil p \rceil-p-1}+|\xi| N^{p-\lceil p \rceil} \;\;\; N\geq 1.
\end{equation}
Furthermore if $p\in \N$ and $\xi \in \R\setminus \Z$ then
\begin{equation}\label{lincorclaim2}
\left|\frac{1}{N} \sum_{n=0}^{N-1} e(\xi n^p) \right| \ll \frac{1}{d(\xi , \Z)} N^{-4^{-p+1}}  \;\;\; N\geq 1.
\end{equation}
\end{corollary}
\begin{proof}
The prior claim \eqref{lincorclaim1} follows from setting $f(x)=\xi x^p$ and $k=\lceil p \rceil$ in \eqref{Corputfk}. The second claim \eqref{lincorclaim2} follows from \eqref{CorputN} with the same choice of $f$.

\end{proof}


\begin{thebibliography}{LMW22}

\bibitem[AFS23]{quantwmafs}
A.~Avila, G.~Forni, and P.~Safaee.
\newblock Quantitative weak mixing for interval exchange transformations.
\newblock {\em Geom. Funct. Anal.}, 33(1):1--56, 2023.

\bibitem[Aus77]{auslander1977lecture}
L.~Auslander.
\newblock {\em Lecture Notes on Nil-Theta Functions}.
\newblock Conference Board of the Mathematical Sciences regional conference
  series in mathematics. Conference Board of the Mathematical Sciences, 1977.

\bibitem[Els13]{elstrodt2013maß}
J.~Elstrodt.
\newblock {\em Ma{\ss}- und Integrationstheorie (German)}.
\newblock Springer-Lehrbuch. Springer Berlin Heidelberg, 2013.

\bibitem[EW10]{einsiedler2010ergodic}
M.~Einsiedler and T.~Ward.
\newblock {\em Ergodic Theory: with a view towards Number Theory}.
\newblock Graduate Texts in Mathematics. Springer London, 2010.

\bibitem[FK17]{forni2017timechanges}
G.~Forni and A.~Kanigowski.
\newblock Time-changes of heisenberg nilflows, 2017.

\bibitem[FU12]{tchorofu}
G.~Forni and C.~Ulcigrai.
\newblock Time-changes of horocycle flows.
\newblock {\em J. Mod. Dyn.}, 6(2):251--273, 2012.

\bibitem[GT12]{taopoly}
B.~Green and T.~Tao.
\newblock The quantitative behaviour of polynomial orbits on nilmanifolds.
\newblock {\em Annals of Mathematics}, 175(2):465--540, 2012.

\bibitem[KM96]{KMbdd}
D.~Y. Kleinbock and G.~A. Margulis.
\newblock Bounded orbits of nonquasiunipotent flows on homogeneous spaces.
\newblock In {\em Sina\u{\i}'s {M}oscow {S}eminar on {D}ynamical {S}ystems},
  volume 171 of {\em Amer. Math. Soc. Transl. Ser. 2}, pages 141--172. Amer.
  Math. Soc., Providence, RI, 1996.

\bibitem[KN74]{kuipers1974uniform}
L.~Kuipers and H.~Niederreiter.
\newblock {\em Uniform Distribution of Sequences}.
\newblock A Wiley-Interscience publication. Wiley, 1974.

\bibitem[KT06]{KATOKspectral}
A.~Katok and J.-P. Thouvenot.
\newblock Chapter 11 - spectral properties and combinatorial constructions in
  ergodic theory.
\newblock In B.~Hasselblatt and A.~Katok, editors, {\em Handbook of Dynamical
  Systems}, volume~1 of {\em Handbook of Dynamical Systems}, pages 649--743.
  Elsevier Science, 2006.

\bibitem[LMW22]{lindenstrauss2022effective}
E.~Lindenstrauss, A.~Mohammadi, and Z.~Wang.
\newblock Effective equidistribution for someone parameter unipotent flows,
  2022.

\bibitem[MS19]{matomaekipoly}
K.~Matomäki and X.~Shao.
\newblock {Discorrelation Between Primes in Short Intervals and Polynomial
  Phases}.
\newblock {\em International Mathematics Research Notices},
  2021(16):12330--12355, 09 2019.

\bibitem[Par04]{parry2004topics}
W.~Parry.
\newblock {\em Topics in Ergodic Theory}.
\newblock Cambridge Tracts in Mathematics. Cambridge University Press, 2004.

\bibitem[Rav22]{ravottiunimix}
D.~Ravotti.
\newblock Polynomial mixing for time-changes of unipotent flows.
\newblock {\em Ann. Sc. Norm. Super. Pisa Cl. Sci. (5)}, 23(3):1491--1506,
  2022.

\bibitem[RW95]{ergharmRW}
J.~M. Rosenblatt and M.~Wierdl.
\newblock Pointwise ergodic theorems via harmonic analysis.
\newblock In {\em Ergodic theory and its connections with harmonic analysis
  ({A}lexandria, 1993)}, volume 205 of {\em London Math. Soc. Lecture Note
  Ser.}, pages 3--151. Cambridge Univ. Press, Cambridge, 1995.

\bibitem[Ven10]{venkateshtwist}
A.~Venkatesh.
\newblock Sparse equidistribution problems, period bounds and subconvexity.
\newblock {\em Ann. of Math. (2)}, 172(2):989--1094, 2010.

\bibitem[Yan23]{yang2023effective}
L.~Yang.
\newblock Effective version of Ratner's equidistribution theorem for
  $\mathrm{SL}(3,\mathbb{R})$, 2023.

\end{thebibliography}
\end{document}